\documentclass[a4paper,reqno]{amsart}

\usepackage{graphicx}
\usepackage{amsmath}
\usepackage{amssymb}
\input xy
\xyoption{all}

\theoremstyle{plain}

\newtheorem{thm}{Theorem}[section]
\newtheorem{cor}[thm]{Corollary}
\newtheorem{lem}[thm]{Lemma}
\newtheorem{prop}[thm]{Proposition}

\theoremstyle{definition}
\newtheorem{defn}[thm]{Definition}
\newtheorem{exm}[thm]{Example}

\theoremstyle{remark}
\newtheorem{rem}[thm]{Remark}

\newdir{ >}{{}*!/-5pt/\dir{>}}

\renewcommand{\mod}{\operatorname{mod}\nolimits}

\newcommand{\id}{\operatorname{id}\nolimits}
\newcommand{\add}{\operatorname{add}\nolimits}
\newcommand{\Hom}{\operatorname{Hom}\nolimits}

\newcommand{\Ker}{\operatorname{Ker}\nolimits}
\newcommand{\Ext}{\operatorname{Ext}\nolimits}

\newcommand{\M}{\mathcal M}

\newcommand{\B}{\mathcal B}
\newcommand{\uB}{\underline{\B}}

\newcommand{\U}{\mathcal U}
\newcommand{\V}{\mathcal V}
\newcommand{\W}{\mathcal W}
\newcommand{\h}{\mathcal H}

\newcommand{\T}{\mathcal T}
\newcommand{\K}{\mathcal K}

\newcommand{\svecv}[2]{\left(\begin{smallmatrix}
      #1 \\
      #2
    \end{smallmatrix}\right)}

\newcommand{\svech}[2]{\left(\begin{smallmatrix}
      #1 & #2
\end{smallmatrix}\right)}

\renewcommand{\emph}{\textit}
\renewcommand{\phi}{\varphi}

\begin{document}

\title{Half exact functors associated with cotorsion pairs on exact categories}
\author{Yu Liu}
\address{Matematiska institutionen \\ Uppsala Universitet \\ Box 480, 751 06 Uppsala, Sweden}
\email{liu.yu@math.uu.se}

\begin{abstract}
In the previous article "Hearts of twin cotorsion pairs on exact categories", we introduced the notion of the heart for any cotorsion pair on an exact category with enough projectives and injectives, and showed that it is an abelian category. In this paper, we construct a half exact functor from the exact category to the heart. This is analog of the construction of Abe and Nakaoka for triangulated categories. We will also use this half exact functor to find out a sufficient condition when two different hearts are equivalent.
\end{abstract}


\maketitle

\section{Introduction}

The important notion of $t$-\emph{structures}, introduced by Beilinson, Bernstein and Deligne \cite{BBD}, is deeply studied in the representation theory. A t-structure in a triangulated category $\T$ is a pair $(\T^{\leq 0},\T^{\geq 0}) $ of subcategories of $\T$, and one of the important properties is that
\begin{itemize}
\item the \emph{heart} $\h=\T^{\leq 0}\cap \T^{\geq 0}$ of the $t$-structure is an abelian category.
\end{itemize}
Moreover, we have the associated cohomological functor $H:\T \to \h$. A typical example is given by the standard $t$-structure $(\T^{\leq 0}, \T^{\geq 0})$ in the derived category $\mathsf D(\mathcal A)$ of an abelian category $\mathcal A$, where $\T^{\leq 0}$ consists of complexes with vanishing cohomologies in positive degrees, and $\T^{\geq 0}$ consists of complexes with vanishing cohomologies in negative degrees. When we have a derived equivalence between abelian categories $\mathcal A$ and $\B$, the we have a new $t$-structure in $\mathsf D(\mathcal A)$ induced by the standard $t$-structure of $\mathsf D(\mathcal B)$. Therefore, $t$-structure is important to study the derived equivalences.

There is a more general notion of \emph{torsion pairs} on triangulated categories, which is a pair $(\U,\V)$ of full subcategories on a triangulated category $\T$ such that
\begin{itemize}
\item $\Hom_\T(\U,\V)=0.$
\item Any object $T\in\T$ admits a triangle $U\rightarrow T\rightarrow V \rightarrow U[1]$ such that $U\in \U$ and $V\in \V$.
\end{itemize}
This notion is classical and has been widely used in the representation theory, since it is useful in various settings to study algebraic structure of triangulated categories.

By a technical reason, we consider a cotorsion pair instead of torsion pair: a pair $(\U,\V)$ on $\T$ is called a \emph{cotorsion pair} if $(\U,\V[1])$ is a torsion pair. Nakaoka introduced the notion of \emph{hearts} of cotorsion pairs on triangulated categories , as a generalization of the heart of $t$-structure, and showed that the hearts are abelian categories \cite{N}.  Abe and Nakaoka constructed a cohomological functor in the case of triangulated categories \cite{AN}, which is a generalization of cohomological functor for t-structure.

Motivated by Nakaoka's results, we will consider cotorsion pairs on Quillen's \emph{exact categories} \cite{Q} (also see \cite{K}), which generalize abelian categories. The cotorsion pairs on abelian categories are ubiquitus in homological algebra. It appears, for instance in tilting theory \cite{AR}, in Cohen-Macaulay representations \cite{AB} and in cluster tilting theory \cite{IY, KR, KZ}. In \cite{L}, we introduced hearts $\underline \h$ of cotorsion pairs $(\U,\V)$ on exact categories $\B$ and proved that they are abelian. In this paper we will construct associated half exact functors $H$ from the exact category $\B$ to the hearts $\underline \h$.

Throughout this paper, let $\B$ be a Krull-Schmidt exact category with enough projectives and injectives. Let $\mathcal P$ (resp. $\mathcal I$) be the full subcategory of projectives (resp. injectives) of $\B$. We recall the definition of a cotorsion pair and its heart on $\B$ \cite[{Definition 2.3}]{L}:

\begin{defn}\label{2}
Let $\U$ and $\V$ be full additive subcategories of $\B$ which are closed under direct summands.
\begin{itemize}

\item[(a)] We call $(\U,\V)$ a \emph{cotorsion pair} if it satisfies the following conditions:
\begin{itemize}

\item[$\bullet$] $\Ext^1_\B(\U,\V)=0$.

\item[$\bullet$] For any object $B\in \B$, there exits two short exact sequences
\begin{align*}
V_B\rightarrowtail U_B\twoheadrightarrow B,\quad
B\rightarrowtail V^B\twoheadrightarrow U^B
\end{align*}
satisfying $U_B,U^B\in \U$ and $V_B,V^B\in \V$.

\end{itemize}

\item[(b)] For any cotorsion pairs $(\U,\V)$, let $\W:=\U\cap \V$. Let
\begin{align*}
\text{ } \B^+:=\{B\in \B \text{ } | \text{ } U_B\in \W \}, \quad
\text{ } \B^-:=\{B\in \B \text{ } | \text{ } V^B\in \W \}, \quad
\h:=\B^+\cap\B^-.
\end{align*}
We call the additive subcategory $\h/\W$ of $\B/\W$ the \emph{heart} of cotorsion pair $(\U,\V)$.

\end{itemize}
\end{defn}

%By definition of a cotorsion pair, we can get:
%\begin{lem}\label{3}
%Let $(\U,\V)$ be a cotorsion pair of $\B$, then
%\begin{itemize}
%\item[(a)] $B$ belongs to $\U$ if and only if $\Ext^1_\B(B,\V)=0$.

%\item[(b)] $B$ belongs to $\V$ if and only if $\Ext^1_\B(\U,B)=0$.

%\item[(c)] $\U$ and $\V$ are closed under extension.

%\item[(d)] $\mathcal P \subseteq \U$ and $\mathcal I \subseteq \V$.

%\end{itemize}
%\end{lem}

Let $\pi:\B \rightarrow \B/\W$ be the canonical quotient functor, for convenience, we denote the ideal quotient of $\B$ by $\W$ as $\uB:=\B/\W$. For any morphism $f\in \Hom_\B(X,Y)$, we denote its image in $ \Hom_{\uB}(X,Y)$ by $\underline f$. For any subcategory $\mathcal C\supseteq\W$ of $\B$, we denote by $\underline {\mathcal C}$ the full subcategory of $\uB$ consisting of the same objects as $\mathcal C$.

We recall the definition of the half exact functor on $\B$ (see e.g. \cite[p.24]{O}).

%We denote by $\Omega: \B/{\mathcal P}\rightarrow \Omega: \B/{\mathcal P}$ the syzygy and by $\Omega^-: \B/{\mathcal I}\rightarrow \Omega^-: \B/{\mathcal I}$

\begin{defn}\label{8.1}
A covariant functor $F$ from $\B$ to an abelian category $\mathcal A$ is called \emph{half exact} if for any short exact sequence
$\xymatrix{A \ar@{ >->}[r]^{f} &B \ar@{->>}[r]^{g} &C}$
in $\B$, %there exists morphisms $h: C\rightarrow \Omega^- A$ and $h': \Omega C\rightarrow A$ such that we can get a long exact sequence
the sequence
$F(A)\xrightarrow{F(f)} F(B)\xrightarrow{F(g)} F(C)$
%\begin{align*}
%\cdots \xrightarrow{F(\Omega h')} F(\Omega A) \xrightarrow{F(\Omega f)} F(\Omega B) \xrightarrow{F(\Omega g)} F(\Omega C) \xrightarrow{F(h')} F(A) \xrightarrow{F(f)} F(B)\\
%\xrightarrow{F(g)} F(C) \xrightarrow{F(h)} F(\Omega^- A) \xrightarrow{F(\Omega^- f)} F(\Omega^- B) \xrightarrow{F(\Omega^- g)} F(\Omega^- C) \xrightarrow{F(\Omega^- h)} \cdots
%\end{align*}
is exact in $\mathcal A$.
\end{defn}

%And when $\B$ is Frobenius, the functor we construct induces a cohomological functor from the stable category of $\B$, which is triangulated by \cite{H}, to the heart of a cotorsion pair.

Let $\K=\add(\U*\V)$, The following is a main theorem of this paper. %we will prove the following theorem (see Theorem \ref{8.11} and Proposition \ref{8.7}, \ref{eqfunctor} for details).

\begin{thm}[Theorem \ref{8.11}, Propositions \ref{8.7}, \ref{eqfunctor}]
Let $(\U,\V)$ be a cotorsion pair on $\B$,
\begin{itemize}

\item[(a)] There exists an associated half exact functor $H:\B\rightarrow \underline \h$ satisfying the following conditions,

\begin{itemize}
\item[$\bullet$] $H|_{\h}=\pi|_{\h}$.
\item[$\bullet$] For any object $X\in \B$, $H(X)=0$ if and only if $X\in \K$.
\end{itemize}

%\item[(c)] For any half exact functor $G:\B\rightarrow \mathcal A$, if $G(\add(\U*\V))=0$, then we have the following commutative diagram
%$$\xymatrix{
%\B \ar[rr]^-H \ar[dr]_G &&{\underline \h} \ar@{.>}[ld]\\
%& {\mathcal A}
%}
%$$
\item[(b)] If a half exact functor $G:\B\rightarrow \underline \h$ satisfies the following conditions,

\begin{itemize}
\item[$\bullet$] $G|_{\h}=\pi|_{\h}$.
\item[$\bullet$] For any object $X\in \B$, $G(X)=0$ if $X\in \K$.
\end{itemize}

then $G \simeq H$.

\end{itemize}
\end{thm}

According to this theorem, we call $\K$ the \emph{kernel} of the associated functor $H$.

%We denote by $\pi_{\mathcal P}:\B\rightarrow \B/{\mathcal P}$ and $\pi_{\mathcal I}:\B\rightarrow \B/{\mathcal I}$ the canonical functors. Then there exist functors $H_{\mathcal P}:\B/{\mathcal P}\rightarrow \underline \h$ and $H_{\mathcal I}:\B/{\mathcal I}\rightarrow \underline \h$ such that $H_{\mathcal P}\pi_{\mathcal P}=H=H_{\mathcal I}\pi_{\mathcal I}$.

We denote by $\Omega: \B/{\mathcal P}\rightarrow \B/{\mathcal P}$ the syzygy functor and by $\Omega^-: \B/{\mathcal I}\rightarrow \B/{\mathcal I}$ the cosyzygy functor. %We will prove in Proposition \ref{8.15} that any half exact functor $F$ which satisfies $F(\mathcal P)=0$ and $F(\mathcal I)=0$ has a similar property as cohomological functors on triangulated categories. In particular, we have the following corollary (see Corollary \ref{8.15} for details).
As an immediate consequence, for any short exact sequence
$\xymatrix{A \ar@{ >->}[r]^{f} &B \ar@{->>}[r]^{g} &C}$
in $\B$, there exist morphisms $h: C\rightarrow \Omega^- A$ and $h': \Omega C\rightarrow A$ such that the sequence
\begin{align*}
\cdots \xrightarrow{H(\Omega h')} H(\Omega A) \xrightarrow{H(\Omega f)} H(\Omega B) \xrightarrow{H(\Omega g)} H(\Omega C) \xrightarrow{H(h')} H(A) \xrightarrow{H(f)} H(B) \quad \quad \quad \quad\\
\xrightarrow{H(g)} H(C) \xrightarrow{H(h)} H(\Omega^- A) \xrightarrow{H(\Omega^- f)} H(\Omega^- B) \xrightarrow{H(\Omega^- g)} H(\Omega^- C) \xrightarrow{H(\Omega^- h)} \cdots
\end{align*}
is exact in $\underline \h$. Moreover, If we have a commutative diagram of short exact sequences
$$\xymatrix@C=0.6cm@R0.6cm{
A \ar@{ >->}[r]^{f} \ar[d]^a &B \ar@{->>}[r]^{g} \ar[d]^b &C \ar[d]^c\\
A_1 \ar@{ >->}[r]^{f_1} &B_1 \ar@{->>}[r]^{g_1} &C_1
}
$$
we can get a commutative diagram of long exact sequences in $\mathcal A$ (see Corollary \ref{8.15} for details).
$$\xymatrix@C=0.6cm@R0.6cm{
\cdots \ar[r] &F(\Omega C) \ar[d]_-{F(\Omega c)} \ar[r]^-{F(h')} &F(A) \ar[r]^{F(f)} \ar[d]^{F(a)} &F(B) \ar[r]^{F(g)} \ar[d]^{F(b)} &F(C) \ar[d]^{F(c)} \ar[r]^-{F(h)}  &F(\Omega^- A) \ar[d]^{F(\Omega^- a)} \ar[r] &\cdots\\
\cdots \ar[r] &F(\Omega C_1)  \ar[r]_-{F(h_1')} &F(A_1) \ar[r]_{F(f_1)} &F(B_1) \ar[r]_{F(g_1)}  &F(C_1)  \ar[r]_-{F(h_1)}  &F(\Omega^- A_1) \ar[r] &\cdots
}
$$

Note that any fixed exact category usually has many different cotorsion pairs, which give us a lot of different hearts. The half exact functor we construct gives us a way to find out the relationship between different hearts. For $k\in \{ 1,2 \}$, $(\U_k,V_k)$ be a cotorsion pair on $\B$, $\W_k= \U_k \cap \V_k$, $\h_k/\W_k$ be the heart of $(\U_k,V_k)$, $H_k$ be the associated half exact functor and $\K_k$ be the kernel of $H_k$. If $\W_1\subseteq \K_2$, then $H_2$ induces a functor $\beta_{12}:\h_1/\W_1\rightarrow \h_2/\W_2$, and we have the following proposition.

\begin{thm}[Propositions \ref{61}, \ref{nameless} and Theorem \ref{serre}]
Let $(\U_1,\V_1)$ and $(\U_2,\V_2)$ be cotorsion pairs on $\B$. If $\W_1\subseteq \K_2\subseteq\K_1$, then
\begin{itemize}

\item[(a)] We have an isomorphism $\beta_{21}\beta_{12}\simeq \id_{\h_1/\W_1}$ of functors from $\h_1/\W_1$ to $\h_2/\W_2$.

\item[(b)] $(\h_2\cap\K_1)/\W_2$ is a Serre subcategory of $\h_2/\W_2$.

\item[(c)] The localization of $\h_2/\W_2$ by the Serre subcategory $(\h_2\cap\K_1)/\W_2$ is equivalent to $\h_1/\W_1$.

\end{itemize}

\end{thm}

This implies the following corollary which gives a sufficient condition when two different hearts are equivalent (see Corollary \ref{suf}).

\begin{cor}
Let $(\U_1,\V_1)$ and $(\U_2,\V_2)$ be cotorsion pairs on $\B$. If $\K_1=\K_2$, then we have an equivalence $\h_1/\W_1\simeq \h_2/\W_2$ between two hearts.
\end{cor}

This result is interesting since this is an analog of derived equivalence.

To construct the associated half exact functor $H$, we first introduce two functors $\sigma^+:\uB\rightarrow \uB^+$ and $\sigma^-:\uB\rightarrow \uB^-$ in section 2, which are analogs of function functors associated with t-structures. In section 3, we show that these two functors commute. We prove the property of the half exact functor in section 4. The functor between different hearts are studied in section 5. The last section contains several examples of our results.

\section{Preliminaries}

We refer to \cite[{\S 2}]{L} and \cite{B} for the details of the exact categories. We introduce the following properties used a lot in this paper, the proofs can be found in \cite[{\S 2}]{B}:

\begin{prop}\label{PO}
Consider a commutative square
$$\xymatrix{
A \;\ar@{>->}[r]^{i} \ar[d]_f &B \ar[d]^{f'}\\
A' \;\ar@{>->}[r]_{i'} &B'
}
$$
in which $i$ and $i'$ are inflations. The following conditions are equivalent:
\begin{itemize}
\item[(a)] The square is a push-out.

\item[(b)] The sequence $\xymatrix{A \;\ar@{>->}[r]^-{\svecv{i}{-f}} &B\oplus A' \ar@{->>}[r]^-{\svech{f'}{i'}} &B'}$ is short exact.

\item[(c)] The square is both a push-out and a pull-back.

\item[(d)] The square is a part of a commutative diagram
$$\xymatrix@C=0.6cm@R0.6cm{
A\; \ar@{>->}[r]^i \ar[d]_f &B \ar[d]^{f'} \ar@{->>}[r] &C \ar@{=}[d]\\
A'\; \ar@{>->}[r]_{i'} &B' \ar@{->>}[r] &C
}$$
with short exact rows.
\end{itemize}
\end{prop}

\begin{prop}\label{7}
\begin{itemize}
\item[(a)] If $\xymatrix{X \ar@{ >->}[r]^{i} &Y \ar@{->>}[r]^{d} &Z}$ and $\xymatrix{N \ar@{ >->}[r]^g &M\ar@{->>}[r]^{f} &Y}$ are two short exact sequences, then there is a commutative diagram of short exact sequences
$$\xymatrix@C=0.6cm@R0.6cm{
N \ar@{ >->}[d] \ar@{=}[r] &N \ar@{ >->}[d]^g\\
Q \ar@{->>}[d] \;\ar@{>->}[r] &M \ar@{->>}[d]^f \ar@{->>}[r] &Z \ar@{=}[d]\\
X \;\ar@{>->}[r]_i &Y \ar@{->>}[r]_d &Z
}$$
where the lower-left square is both a push-out and a pull-back.

\item[(b)] If $\xymatrix{X \ar@{ >->}[r]^{i} &Y \ar@{->>}[r]^{d} &Z}$ and $\xymatrix{Y \ar@{ >->}[r]^{g} &K \ar@{->>}[r]^f &L}$ are two short exact sequences, then there is a commutative diagram of short exact sequences
$$\xymatrix@C=0.6cm@R0.6cm{
X \ar@{=}[d] \ar@{ >->}[r]^{i} &Y \ar@{ >->}[d]^g \ar@{->>}[r]^{d} &Z \ar@{ >->}[d]\\
X \ar@{ >->}[r]  &K \ar@{->>}[r] \ar@{->>}[d]^f &R \ar@{->>}[d]\\
&L \ar@{=}[r] &L
}$$
where the upper-right square is both a push-out and a pull-back.
\end{itemize}
\end{prop}

We recall some important definitions and results of \cite{L}, which also work for a single cotorsion pair.

\begin{defn}\label{re}
For any $B\in\B$, we define $B^+$ and $\alpha_B:B\rightarrow B^+$ as follows:\\
Take two short exact sequences:
\begin{align*}
\xymatrix{V_B \ar@{ >->}[r] &U_B \ar@{->>}[r]^{u_B} &B},\quad
\xymatrix{U_B \ar@{ >->}[r]^{w'} &W^0 \ar@{->>}[r] &U^0}
\end{align*}
where $U_B,U^0 \in \U$, $W^0$,$\V_B\in \V$. In fact, $W^0\in \W$ since $\U$ is closed under extension. By Proposition \ref{7}, we get the following commutative diagram
\begin{equation}\label{F1}
$$\quad \quad \quad \quad \quad \quad \quad \quad \quad \quad \quad \quad \quad \quad \quad \quad \quad \quad\xymatrix@C=0.6cm@R0.6cm{
V_B \ar@{=}[d] \ar@{ >->}[r] &U_B \ar@{ >->}[d]_{w'} \ar@{->>}[r]^{u_B} &B \ar@{ >->}[d]^{\alpha_B}\\
V_B \ar@{ >->}[r]  &W^0 \ar@{->>}[r]_w \ar@{->>}[d] &B^+ \ar@{->>}[d]\\
&U^0 \ar@{=}[r] &U^0
}$$
\end{equation}
where the upper-right square is both a push-out and a pull-back.
\end{defn}

By definition, $B^+\in \B^+$. We recall the following useful proposition.

\begin{prop}\cite[{Lemma 3.2, }{Proposition 3.3}]{L}\label{8}
For any $B\in \B$
\begin{itemize}
\item[(a)] If $B\in \B^-$, then $B^+\in \h$.

\item[(b)] $\alpha_B$ is a left $\B^+$-approximation, and for an object $Y\in \B^+$, $\Hom_{\uB}(\underline {\alpha_B},Y):\Hom_{\uB}(B^+,Y)\rightarrow \Hom_{\uB}(B,Y)$ is bijective.
\end{itemize}
\end{prop}

By Proposition \ref{8}, we can define a functor $\sigma^+$ from $\uB$ to $\underline \B^+$ as follows:\\
For any object $B\in \B$, since all the ${B^+}$ are isomorphic to each other in $\uB$ by Proposition \ref{8}, we fix a $B^+$ for $B$. Let
\begin{align*}
\sigma^+:\uB \rightarrow \underline \B^+\\
B \mapsto B^+
\end{align*}
and for any morphism $\underline f:B\rightarrow C$, we define $\sigma^+(\underline f)$ as the unique morphism given by Proposition \ref{8}
$$\xymatrix{
B \ar[r]^{\underline f} \ar[d]_{\underline {\alpha_B}} &C \ar[d]^{\underline {\alpha_C}}\\
B^+ \ar@{.>}[r]_{\sigma^+(\underline f)} &C^+.
}
$$

Let $i^+:\uB^+\hookrightarrow \uB$ be the inclusion functor. Then $(\sigma^+,i^+)$ is an adjoint pair by Proposition \ref{8}.

\begin{prop}\label{8.2}
The functor $\sigma^+$ has the following properties:
\begin{itemize}
\item[(a)] $\sigma^+$ is an additive functor.
%\item[(a)] For any objects $A$ and $B$ in $\uB$, $\sigma^+(A\oplus B) \simeq \sigma^+(A)\oplus \sigma^+(B)$ in $\uB$.
\item[(b)] $\sigma^+|_{\uB^+}=\id_{\uB^+}$.
\item[(c)] For any morphism $f:A\rightarrow B$, $\sigma^+(\underline f)=0$ in $\uB$ if and only if $f$ factors through $\U$. In particular, $\sigma^+(B)=0$ if and only if $B\in \underline \U$.
\end{itemize}
\end{prop}

\begin{proof}
(a), (b) can be concluded easily by definition.\\
(c) The "if" part is followed by \cite[{Lemma 3.4}]{L}.\\
Now suppose $\sigma^+(\underline f)=0$ in $\uB$. By Proposition \ref{8}, we have the following commutative diagram
$$\xymatrix@C=0.6cm@R0.6cm{
A \ar@{ >->}[d]_{\alpha_A} \ar[r]^f &B \ar@{ >->}[d]^{\alpha_B} &U_B \ar@{ >->}[d]^{w'} \ar@{->>}[l]_{u_B} &V_B \ar@{=}[d] \ar@{ >->}[l]\\
A^+ \ar@{->>}[d] \ar[r]_{f^+} &B^+ \ar@{->>}[d] &W^0 \ar@{->>}[d] \ar@{->>}[l]^w &V_B \ar@{ >->}[l]\\
U^0_A \ar[r] &U^0 \ar@{=}[r] &U^0
}$$
where $\underline f^+=\sigma^+(\underline f)$. Then $f^+$ factors through an object $W\in \W$.
$$\xymatrix@C=0.6cm@R0.6cm{
A^+ \ar[dr]_a \ar[rr]^{f^+} &&B^+\\
&W \ar[ur]_b
}$$
Since $w$ is a right $\U$-approximation of $B^+$, there exists a morphism $c:W\rightarrow W^0$ such that $b=wc$. Thus $\alpha_Bf=f^+\alpha_A=ba\alpha_A=w(ca\alpha_A)$. By the definition of pull-back, there exists a morphism $d:A\rightarrow U_B$ such that $f=u_Bd$. Thus $f$ factors through $\U$.
\end{proof}

\begin{defn}\label{10}
For any object $B\in \B$, we define $B^-$ and $\gamma_B:B^- \rightarrow B$ as follows:

Take the following two short exact sequences
\begin{align*}
\xymatrix{B \ar@{ >->}[r]^{v^B} &V^B \ar@{->>}[r] &U^B},\quad
\xymatrix{V_0 \ar@{ >->}[r] &W_0 \ar@{->>}[r] &V^B}
\end{align*}
where $V^B,V_0\in \V$, and $W_0$,$U^B\in \U$. Then $W_0\in \W$ holds since $\V$ is closed under extension. By Proposition \ref{7}, we get the following commutative diagram:
\begin{equation}\label{F2}
$$\quad \quad \quad \quad \quad \quad \quad \quad \quad \quad \quad \quad \quad \quad \quad \quad \quad \quad\xymatrix@C=0.6cm@R0.6cm{
V_0 \ar@{ >->}[d]_v \ar@{=}[r] &{V_0} \ar@{ >->}[d]\\
B^- \ar@{->>}[d]_{\gamma_B} \ar@{ >->}[r] &{W_0} \ar@{->>}[d] \ar@{->>}[r] &{U^B} \ar@{=}[d]\\
B \ar@{ >->}[r]_{v^B} &{V^B} \ar@{->>}[r] &{U^B}.}
$$
\end{equation}
\end{defn}

By definition $B^-\in \B^-$ and we have:

\begin{prop}\cite[{Proposition 3.6}]{L}\label{81}
For any object $B\in \B$
\begin{itemize}

\item[(a)] $B\in \B^+$ implies $B^- \in \h$.

\item[(b)] $\gamma_B$ is a right $\B^-$-approximation. For any $X\in \B^-$, $\Hom_{\uB}( X,\underline{\gamma_B}):\Hom_{\uB}(X,B^-)\rightarrow \Hom_{\uB}(X,B)$ is bijective.
\end{itemize}
\end{prop}

We define a functor $\sigma^-$ from $\uB$ to $\underline \B^-$ as the dual of $\sigma^+$:
\begin{align*}
\sigma^-:\uB \rightarrow \underline \B^-\\
B \mapsto B^-.
\end{align*}
For any morphism $\underline f:B\rightarrow C$, we define $\sigma^-(\underline f)$ as the unique morphism given by Proposition \ref{81}
$$\xymatrix{
B^- \ar@{.>}[r]^{\sigma^-(\underline f)} \ar[d]_{\underline {\gamma_B}} &C^- \ar[d]^{\underline {\gamma_C}}\\
B \ar[r]_{\underline f} &C.
}
$$

Let $i^-:\uB^-\hookrightarrow \uB$ be the inclusion functor, then $(i^-,\sigma^-)$ is an adjoint pair by Proposition \ref{81}.

\begin{prop}\label{8.3}
The functor $\sigma^-$ has the following properties:
\begin{itemize}
\item[(a)] $\sigma^-$ is an additive functor.
%\item[(a)] For any objects $A$ and $B$ in $\uB$, $\sigma^-(A\oplus B) \simeq \sigma^-(A)\oplus \sigma^-(B)$ in $\uB$.
\item[(b)] $\sigma^-|_{\uB^-}=\id_{\uB^-}$.
\item[(c)] For any morphism $f:A\rightarrow B$, $\sigma^-(\underline f)=0$ in $\uB$ if and only if $f$ factors through $\V$. In particular, $\sigma^-(B)=0$ if and only if $B\in \underline \V$.
\end{itemize}
\end{prop}

\section{Reflection sequences and coreflection sequences}
In the following two sections we fix a cotorsion pair $(\U,\V)$. In \cite{AN}, the reflection (resp. coreflection) sequences are introduced on triangulated categories,  We will define the similar sequences on exact categories.

Let $\mathcal C$ be a subcategory of $\B$, denote by $\Omega \mathcal C$ (resp. $\Omega^- \mathcal C$) the subcategory of $\B$ consisting of objects $\Omega C$ (resp. $\Omega^- C$) such that there exists a short exact sequence
$$\Omega C\rightarrowtail P_C\twoheadrightarrow C \text{ } (P_C\in \mathcal P, C\in \mathcal C) \text{ }
(\text{resp. } C\rightarrowtail I^C\twoheadrightarrow \Omega^- C \text{ } (I^C\in \mathcal I, C\in \mathcal C))$$

%Since for any $A\in \B$, $\Omega A$ is unique up to isomorphisms in $\B/\mathcal P$, we fix a short exact sequence for each $A$

\begin{lem}\label{8.00}
$\Omega \U\subseteq \B^-$ and $\Omega^- \V\subseteq \B^+$.
\end{lem}

\begin{proof}
We only prove the first one, the second is dual.\\
For any $U\in \U$, we consider two short exact sequences
$$\xymatrix{
\Omega U \ar@{ >->}[r]^q &P_{U} \ar@{->>}[r] &U,\\
} \quad \xymatrix{
\Omega U \ar@{ >->}[r]^{v'} &V^{\Omega U} \ar@{->>}[r] &U^{\Omega U}
}
$$
with $P_U\in \mathcal P$, $U^{\Omega U}\in \U$ and $V^{\Omega U}\in \V$. Since $\Ext^1_\B(U,V^{\Omega U})=0$, there exists a morphism $p:P_U\rightarrow V^{\Omega U}$ such that $pq=v'$, hence we get a commutative diagram
$$\xymatrix@C=0.6cm@R0.6cm{
\Omega U \ar@{ >->}[r]^q \ar@{=}[d] &P_U \ar[d]^p \ar@{->>}[r] &U \ar[d]\\
\Omega U \ar@{ >->}[r]_{v'} &V^{\Omega U} \ar@{->>}[r] &U^{\Omega U}
}
$$
which induces a short exact sequence $\xymatrix{P_U \ar@{ >->}[r] &V^{\Omega U}\oplus U \ar@{->>}[r] &{U.}^{\Omega U}}$ Since $\U$ is closed under extension and direct summands, $V^{\Omega U}\in \U$. Thus $\Omega U\in \B^-$.
%It is enough to show that $V'\in \U$. Apply $\Hom_\B(-,\V)$ to the above diagram, we get the following commutative diagram
%$$\xymatrix{
%0=\Ext^1_\B(U',\V) \ar[r] \ar[d] &\Ext^1_\B(V',\V) \ar[rr]^{\Ext^1_\B(v',\V)} \ar[d] &&\Ext^1_\B(\Omega U,\V) \ar@{=}[d]\\
%0=\Ext^1_\B(U,\V) \ar[r] &\Ext^1_\B(P_U,\V) \ar[rr]^{\Ext^1_\B(q,\V)} &&\Ext^1_\B(\Omega U,\V).
%}
%$$
%Since $\Ext^1_\B(P_U,\V)=0$, we obtain $\Ext^1_\B(q,\V)=0$, hence $\Ext^1_\B(v',\V)=0$, which implies that $\Ext^1_\B(V',\V)=0$. Thus $V'\in \U$, which implies that $\Omega U\in \B^-$.
%By Definition \ref{10}, we have the following commutative diagram:
%$$\xymatrix{
%V'' \ar@{ >->}[d] \ar@{=}[r] &V'' \ar@{ >->}[d]\\
%\sigma^-(\Omega U) \ar@{ >->}[r] \ar@{}[dr]|{PB} \ar@{->>}[d] &W \ar@{->>}[r] \ar@{->>}[d]^f &U' \ar@{=}[d]\\
%\Omega U \ar@{ >->}[r]_{v'} &V' \ar@{->>}[r] &U'.
%}
%$$
%Since $P_U$ is projective, there exists a morphism $g:P_U\rightarrow W$ such that $fg=p$. Hence by the definition of pull-back, we get the following commutative diagram:
%$$\xymatrix{
%\Omega U \ar@{.>}[dr] \ar@{=}[ddr] \ar[rrr]^q &&&P_U \ar[dl]_-g \ar[ddl]^-p\\
%&\sigma^-(\Omega U) \ar[r] \ar[d] &W \ar[d]_f\\
%&\Omega U \ar[r]_{v'} &V'
%}
%$$
%which implies that $\Omega U$ is a direct summand of $\sigma^-(\Omega U)\in \B^-$, then by \cite[{Lemma 2.9}]{L}, we get $\Omega U\in \B^-$.
\end{proof}

By this lemma and \cite[Lemma 2.10,2.11]{L}, we get the following Corollary.

\begin{cor}
\begin{itemize}
\item[(a)] In the short exact sequence $\Omega U\rightarrowtail P_U\twoheadrightarrow U$ where $U\in \U$ and $P_U\in \mathcal P$, we have $P_U\in\W$.

\item[(b)] In the short exact sequence $V\rightarrowtail I^V\twoheadrightarrow \Omega^- V$ where $V\in \V$ and $I^V\in \mathcal I$, we have $I^V\in \W$.
\end{itemize}
\end{cor}

%\begin{lem}\label{8.01}
%Let $A$, $C$ be any two objects in $\B$. Then we have $\Hom_{\B/\mathcal P}(\Omega C, A)\simeq \Hom_{\B/\mathcal I}(C, \Omega^- A)$.
%\end{lem}

%\begin{proof}
%Since $A$ and $C$ admit the following two short exact sequences
%$$\xymatrix{\Omega C \ar@{ >->}[r]^{q_C} &P_C \ar@{->>}[r]^{p_C} &C,} \quad \xymatrix{A \ar@{ >->}[r]^{i^A} &I^A \ar@{->>}[r]^{j^A} &\Omega^- A}$$
%for any morphism $f:\Omega C\rightarrow A$, it induces a commutative diagram
%$$\xymatrix{
%\Omega C \ar[d]_f \ar@{ >->}[r]^{q_C} &P_C \ar[d]^g \ar@{->>}[r]^{p_C} &C \ar[d]^h\\
%A \ar@{ >->}[r]_{i^A} &I^A \ar@{->>}[r]_{j^A} &\Omega^- A.
%}
%$$
%We claim that the image of $f$ in $\B/\mathcal P$ and the image of $h$ in $\B/\mathcal I$ are one-to-one correspondent.\\
%If $f$ induces another commutative diagram
%$$\xymatrix{
%\Omega C \ar[d]_f \ar@{ >->}[r]^{q_C} &P_C \ar[d]^{g'} \ar@{->>}[r]^{p_C} &C \ar[d]^{h'}\\
%A \ar@{ >->}[r]_{i^A} &I^A \ar@{->>}[r]_{j^A} &\Omega^- A
%}
%$$
%then we can conclude easily that $h-h'$ factors through $I^A$, which implies that $h$ is unique in $\B/\mathcal I$.\\
%Dually, any triple $(f'',g'',h)$ which makes the following diagram
%$$\xymatrix{
%\Omega C \ar[d]_{f''} \ar@{ >->}[r]^{q_C} &P_C \ar[d]^{g''} \ar@{->>}[r]^{p_C} &C \ar[d]^h\\
%A \ar@{ >->}[r]_{i^A} &I^A \ar@{->>}[r]_{j^A} &\Omega^- A
%}
%$$
%commute satisfies that $f''-f$ factors through in $P_C$. Since $P_C\in \mathcal P$, for any morphism $h:C\rightarrow \Omega^- A$, there exists a triple $(f'',g'',h)$ which makes the above diagram commute.
%\end{proof}

\begin{defn}\label{8.4}
Let $B$ be any object in $\B$.
\begin{itemize}
\item[(a)] A \emph{reflection sequence} for $B$ is a short exact sequence
$$\xymatrix{B \ar@{ >->}[r]^{z} &Z \ar@{->>}[r] &U}$$
where $U\in \U$, $Z\in \B^+$ and there exists a commutative diagram
$$\xymatrix@C=0.6cm@R0.6cm{
\Omega U \ar@{ >->}[r]^q \ar[d]_x &P_U \ar@{ ->>}[r] \ar[d]^p &U \ar@{=}[d]\\
B \ar@{ >->}[r]_z &Z \ar@{->>}[r] &U}
$$
with $P_U\in \mathcal P$ and $x$ factoring through $\U$.

\item[(b)] A \emph{coreflection sequence} for $B$ is a short exact sequence
$$\xymatrix{V \ar@{ >->}[r] &K \ar@{->>}[r]^{k} &B}$$
where $V\in \V$, $K\in \B^-$ and there exists a commutative diagram
$$\xymatrix@C=0.6cm@R0.6cm{
V \ar@{ >->}[r] \ar@{=}[d] &K \ar@{->>}[r]^k \ar[d] &B \ar[d]^y\\
V \ar@{ >->}[r] &I^V \ar@{->>}[r] &\Omega^-V}
$$
with $I^V\in \mathcal I$ and $y$ factoring through $\V$.
\end{itemize}
\end{defn}

\begin{lem}\label{8.5}
Let $B$ be an object in $\B$. Then
\begin{itemize}
\item[(a)] The short exact sequence $\xymatrix{B \ar@{ >->}[r]^{\alpha_B} &B^+ \ar@{->>}[r] &U^0}$ in (2.1) is a reflection sequence for $B$.

\item[(b)] The short exact sequence $\xymatrix{V_0 \ar@{ >->}[r] &B^- \ar@{->>}[r]^{\gamma_B} &B}$ in (2.2) is a coreflection sequence for $B$.

\item[(c)] For any reflection sequence $\xymatrix{B \ar@{ >->}[r]^{z} &Z \ar@{->>}[r] &U}$ for $B$, we have $Z\simeq B^+$ in $\uB$.

\item[(d)] For any coreflection sequence $\xymatrix{V \ar@{ >->}[r] &K \ar@{->>}[r]^{k} &B}$ for $B$, we have $K\simeq B^-$ in $\uB$.

%\item For any reflection sequence $B\rightarrowtail Z\twoheadrightarrow U$ for $B$, whenever we form a commutative diagram
%$$\xymatrix{
%\Omega U \ar@{ >->}[r] \ar[d]_x &P_U \ar@{->>}[r] \ar[d] &U \ar@{=}[d]\\
%B \ar@{ >->}[r] &Z \ar@{->>}[r] &U}
%$$
%with rows short exact and $P_U\in \mathcal P$, we get $x$ factors through $\U$.

%\item For any coreflection sequence $V\rightarrowtail K\twoheadrightarrow B$ for $B$, whenever we form a commutative diagram
%$$\xymatrix{
%V \ar@{ >->}[r] \ar@{=}[d] &K \ar@{->>}[r] \ar[d] &B \ar[d]^y\\
%V \ar@{ >->}[r] &I_V \ar@{->>}[r] &\Omega^-V}
%$$
%with rows short exact and $I_V\in \mathcal I$, we get $y$ factors through $\V$.
\end{itemize}
\end{lem}

\begin{proof}
We only prove (a) and (c), the other two are dual.\\
(a) Since $U^0$ admits the following short exact sequence
$\xymatrix{\Omega U^0 \ar@{ >->}[r]^{q_0} &P_{U^0} \ar@{->>}[r] &U^0}$,
we get the following commutative diagram
$$\xymatrix@C=0.6cm@R0.6cm{
\Omega U^0 \ar@{ >->}[r]^{q_0} \ar[d]_{x_0} &P_{U^0} \ar@{->>}[r] \ar[d]^{p_0} &U^0 \ar@{=}[d]\\
B \ar@{ >->}[r]_{\alpha_B} &B^+ \ar@{->>}[r] &U^0.
}
$$
Since $P_{U^0}$ is projective, there exists a morphism $p'_0:P_{U^0}\rightarrow W^0$ such that $wp'_0=p_0$, we get $\alpha_Bx_0=p_0q_0=wp'_0q_0$. Then $x_0$ factors through $U_B\in \U$ since (2.1) is a pull-back diagram.
$$\xymatrix{
&\Omega U^0 \ar@{.>}[ddl] \ar@{ >->}[r]^{q_0} \ar[d]^{x_0} &P_{U^0} \ar[ddl]_(.25){p_0'} \ar@{ ->>}[r] \ar[d]^{p_0} &U^0 \ar@{=}[d]\\
&B \ar@{ >->}[r]^(.35){\alpha_B} &B^+ \ar@{->>}[r] &U^0\\
U_B \ar@{ >->}[r]_{w'} \ar[ur]_{u_B} &W^0 \ar[ur]_w \ar@{->>}[r] &U^0 \ar@{=}[ur]
}
$$
Hence by definition $\xymatrix{B \ar@{ >->}[r]^{\alpha_B} &B^+ \ar@{->>}[r] &U^0}$ is a reflection sequence for $B$.\\
(c) We first show that there exists a morphism $\underline f: Z\rightarrow B^+$ such that $\alpha_B=fz$.\\
The reflection sequence admits a commutative diagram
$$\xymatrix@C=0.6cm@R0.6cm{
\Omega U \ar@{ >->}[r]^q \ar[d]_x &P_U \ar@{ ->>}[r] \ar[d]^p &U \ar@{=}[d]\\
B \ar@{ >->}[r]_z &Z \ar@{->>}[r] &U}
$$
where the left square is a push-out by Proposition 2.1. Since $x$ factors through $\U$, and $u_B$ is a right $\U$-approximation of $B$, there exists a morphism $x':\Omega U\rightarrow U_B$ such that $x=u_Bx'$.
%we get the following commutative diagram:
%$$\xymatrix{
%&U_B \ar[dr]^{u_B}\\
%\Omega U \ar[ur]^y \ar[rr]_x &&B
%}
%$$
Since $\Ext^1_\B(U,W^0)=0$, there exists a morphism $p':P_U\rightarrow W^0$ such that $w'x'=p'q$, thus $\alpha_Bx=\alpha_Bu_Bx'=ww'x'=wp'q$. Then by the definition of push-out, there exists a morphism $f:Z\rightarrow B^+$ such that $\alpha_B=fz$.
$$\xymatrix{
&\Omega U \ar[dddl]_{x'} \ar@{ >->}[r]^q \ar[d]^x &P_U \ar[dddl]^{p'} \ar@{ ->>}[r] \ar[d]^p &U \ar@{=}[d]\\
&B \ar@{ >->}[r]^z \ar@{=}[d] &Z \ar@{->>}[r]^a \ar@{.>}[d]^f &U\\
&B \ar@{ >->}[r]_{\alpha_B} &B^+ \ar@{->>}[r] &U^0\\
U_B \ar@{ >->}[r]_{w'} \ar[ur]_{u_B} &W^0 \ar[ur]_w \ar@{->>}[r] &U^0 \ar@{=}[ur]
}
$$
By Proposition \ref{8}, there is a morphism $g:B^+\rightarrow Z$ such that $g\alpha_B=z$, we have a morphism $\underline {fg}:B^+\rightarrow B^+$ such that $\underline {fg\alpha}=\underline \alpha$, which implies that $\underline {fg}=\underline \id_{B^+}$. \\
Now we prove that $\underline {gf}=\underline \id_Z$.\\
Since $(gf-\id_Z)z=0$, we get a morphism $b:U\rightarrow B^+$ such that $gf-\id_Z=ba$. Since $\Ext^1_\B(U,V_B)=0$, $b$ factors through $W^0$, hence $\underline {gf}=\underline \id_Z$.\\
Thus $B^+\simeq Z$ in $\uB$.
\end{proof}

%We need the following lemma, the proof is similar as \cite[{Proposition 3.3}]{L}.

%\begin{lem}\label{8.12}
%Let $B$ be any object in $\B$. Then
%\begin{itemize}
%\item[(a)] $\Hom_\B(\alpha_B,\Omega^-\V):\Hom_\B(B^+,\Omega^-\V)\rightarrow \Hom_\B(B,\Omega^-\V)$ is surjective.

%\item If $\mathcal I\in \W$, then $\Hom_{\uB}(\underline z_B,\Omega'\V):\Hom_{\uB}(Z_B,\Omega'\V)\rightarrow \Hom_{\uB}(B,\Omega'\V)$ is bijective.

%\item[(b)] $\Hom_\B(\Omega \U,\gamma_B):\Hom_\B(\Omega \U,B^-)\rightarrow \Hom_\B(\Omega \U,B)$ is surjective.

%\item If $\mathcal P\in\W$, then $\Hom_{\uB}(\Omega \U,\underline k_B):\Hom_{\uB}(\Omega \U,K_B)\rightarrow \Hom_{\uB}(\Omega \U,B)$ is bijective.
%\end{itemize}
%\end{lem}

%We need the following proposition:
%\begin{lem}\cite[{Corollary 3.2}]{B}\label{8.13}
%Consider a morphism of short exact sequence
%$$\xymatrix{
%A' \ar@{ >->}[r] \ar[d]^a &B' \ar@{->>}[r] \ar[d]^b &C' \ar[d]^c\\
%A \ar@{ >->}[r] &B \ar@{->>}[r] &C.
%}
%$$
%If $a$ and $c$ are isomorphisms (or inflations, or deflations) then so is $b$.
%\end{lem}

\begin{prop}\label{8.6}
There exists an isomorphism of functors from $\uB$ to $\underline \h$
$$\eta:\sigma^+ \circ \sigma^-\xrightarrow{\simeq}\sigma^- \circ \sigma^+.$$
\end{prop}

\begin{proof}
By Proposition \ref{8} and \ref{81} both $\sigma^+ \circ \sigma^-$ and $\sigma^- \circ \sigma^+$ are functors from $\uB$ to $\underline \h$.\\
By Lemma \ref{8.5}, We can take the following commutative diagram of short exact sequences
$$\xymatrix@C=0.6cm@R0.6cm{
V_0 \ar@{ >->}[r]^v \ar@{=}[d] &B^- \ar@{->>}[r]^{\gamma_B} \ar[d]_d &B \ar[d]^{y_0}\\
V_0 \ar@{ >->}[r]_j &I^0 \ar@{->>}[r]_i &\Omega^-V_0
}
$$
where $y_0$ factors through $V^B$ since $v^B$ is a left $\V$-approximation of $\B$.
$$\xymatrix@C=0.6cm@R0.6cm{
B \ar[rr]^{y_0} \ar[dr]_{v^B} &&\Omega^-V_0\\
&V^B \ar[ur]_{v'}
}
$$
By Lemma \ref{8.00} and Proposition \ref{8}, there exists a morphism $t:B^+\rightarrow \Omega^-V_0$ such that $y_0=t\alpha_B$. Since $\Ext^1_\B(U^0,V^B)=0$, there exists a morphism $v_0:B^+\rightarrow V^B$ such that $v^B=v_0\alpha_B$. Thus $t\alpha_B=v'v^B=v'v_0\alpha_B$, then we obtain that $t-v'v_0$ factors through $U^0$.
$$\xymatrix{
B \ar@{ >->}[r]^{\alpha_B} \ar[d]_{v^B} &B^+ \ar[dl]^{v_0}_{\circlearrowright} \ar[d]^-{t-v'v_0} \ar@{->>}[rr] &&U^0 \ar@{.>}@/^/[dll]^u_{\circlearrowright}\\
V^B \ar[r]_{v'} &\Omega^-V_0
}
$$
Since $\Ext^1_\B(U^0,V_0)=0$, $u$ factors through $I^0\in \V$. Hence $t$ factors through $\V$.\\
Take a pull-back of $t$ and $i$, we get the following commutative diagram
$$\xymatrix@C=0.6cm@R0.6cm{
V_0 \ar@{ >->}[r] \ar@{=}[d] &Q \ar@{}[dr]|{PB} \ar@{->>}[r]^s \ar[d]_{d'} &B^+ \ar[d]^t\\
V_0 \ar@{ >->}[r]_j &I^0 \ar@{->>}[r]_i &\Omega^-V_0.
}
$$
By \cite[{Lemma 2.11}]{L}, we obtain $Q\in \B^+$. Now by Proposition \ref{7}, we get the following commutative diagram
$$\xymatrix@C=0.6cm@R0.6cm{
V_0 \ar@{ >->}[d] \ar@{=}[r] &V_0 \ar@{ >->}[d]\\
Q' \ar@{ >->}[r] \ar@{->>}[d] &Q \ar@{->>}[r] \ar@{->>}[d]^s &U^0 \ar@{=}[d]\\
B \ar@{ >->}[r]_{\alpha_B} &B^+ \ar@{->>}[r] &U^0.
}
$$
By the definition of pull-back, there exists a morphism $k:B\rightarrow Q$ such that $sk=\alpha_B\gamma_B$ and $d'k=d$. Hence we have the following diagram
$$\xymatrix{
V_0 \ar@/_15pt/[dd]_{\text{id}_{V_0}} \ar@{.>}[d]^{v_0} \ar@{ >->}[r]^v &B^- \ar@{.>}[d]^k \ar@/_15pt/[dd]_(.75)d \ar@{->>}[r]^{\gamma_B} &B \ar[d]^{\alpha_B}\\
V_0 \ar@{ >->}[r] \ar@{=}[d] &Q \ar@{->>}[r]^s \ar@{}[dr]|{PB} \ar[d]^{d'} &B^+ \ar[d]^t\\
V_0 \ar@{ >->}[r]_j &I^0 \ar@{->>}[r]_i &\Omega^-V_0
}
$$
where the upper-left square commutes. Hence $jv_0=d'kv=dv=j$, we can conclude that $v_0=\text{id}_{V_0}$ since $j$ is monomorphic.
By the same method we can get the following commutative diagram
$$\xymatrix{
V_0 \ar@/_15pt/[dd]_{\text{id}_{V_0}} \ar@{.>}[d]^{v_0'} \ar@{ >->}[r]^v &B^- \ar@{.>}[d]^{k'} \ar@/_15pt/[dd]_(.75)k \ar@{->>}[r]^{\gamma_B} &B \ar@{=}[d]\\
V_0 \ar@{ >->}[r] \ar@{=}[d] &Q' \ar@{->>}[r] \ar@{ >->}[d] &B \ar@{ >->}[d]^{\alpha_B}\\
V_0 \ar@{ >->}[r] &Q \ar@{->>}[r] &B^+
}
$$
where $v_0'=\text{id}_{V_0}$. Therefore $k'$ is isomorphic by \cite[{Corollary 3.2}]{B}. We obtain the following commutative diagram
$$\xymatrix@C=0.6cm@R0.6cm{
V_0 \ar@{ >->}[d]_v \ar@{=}[r] &V_0 \ar@{ >->}[d]\\
B^- \ar@{ >->}[r]^k \ar@{->>}[d]_{\gamma_B} &Q \ar@{->>}[r] \ar@{->>}[d]^s &U^0 \ar@{=}[d]\\
B \ar@{ >->}[r]_{\alpha_B} &B^+ \ar@{->>}[r] &U^0.
}
$$
%$$\xymatrix{V_0 \ar@{ >->}[r] &Q' \ar@{->>}[r] &B} \quad \text{by} \quad \xymatrix{V_0 \ar@{ >->}[r]^v &B^- \ar@{->>}[r]^{\gamma_B} &B.}$$
%Now we have a short exact sequence $\xymatrix{B^- \ar@{ >->}[r]^{k} &Q \ar@{->>}[r] &U^0}$.
We get $Q\in \B^-$ by \cite[{Lemma 2.10}]{L}, hence $Q\in \h$.
Since $t$ factors through $\V$, $\xymatrix{V_0 \ar@{ >->}[r] &Q \ar@{->>}[r]^{s} &B^+}$ is a coreflection sequence for $B^+$. By Lemma \ref{8.5}, we have the following commutative diagram
$$\xymatrix@C=0.6cm@R0.6cm{
&Q \ar[dr]^{\underline s} \ar@{.>}[dl]_-{\alpha'}\\
\sigma^-(B^+) \ar[rr]_\alpha &&B^+
}
$$
in $\uB$ where $\alpha'$ is isomorphic. %Thus $Q\simeq \sigma^-(B^+)=\sigma^-\sigma^+(B)$ in $\underline \h$.\\

By duality we conclude that $\xymatrix{B^- \ar@{ >->}[r]^{k} &Q \ar@{->>}[r] &U^0}$ is a reflection sequence for $B^-$. By Lemma \ref{8.5}, we have the following commutative diagram
$$\xymatrix@C=0.6cm@R0.6cm{
B^- \ar[rr]^{\beta} \ar[dr]_{\underline k} &&\sigma^+(B^-) \ar@{.>}[dl]^{\beta'}\\
&Q
}
$$
in $\uB$ where $\beta'$ is isomorphic. %Thus $Q\simeq \sigma^+(B^-)=\sigma^+\sigma^-(B)$ in $\underline \h$.\\

By Proposition \ref{81}, there exists a morphism $\theta:B^-\rightarrow \sigma^-\sigma^+(B)$ in $\uB$ such that $\alpha \theta=\underline {\alpha_B\gamma_B}$. Then by Proposition \ref{8}, there exists a unique morphism $\eta_B:\sigma^+\sigma^-(B) \rightarrow \sigma^-\sigma^+(B)$ such that $\eta_B\beta=\theta$. Hence we get the following commutative diagram
$$\xymatrix{
\sigma^+\sigma^-(B) \ar@{.>}[dd]_{\eta_B} &B^- \ar[l]_-{\beta} \ar@{.>}[ddl]^\theta \ar[d]^{\underline {\gamma_B}}\\
&B \ar[d]^{\underline {\alpha_B}}\\
\sigma^-\sigma^+(B) \ar[r]_-{\alpha} &B^+.
}
$$
Then $\alpha\eta_B\beta=\underline {\alpha_B\gamma_B}=\underline {sk}=\alpha\alpha'\beta'\beta$, and we have $\eta_B=\alpha'\beta'$ by Proposition \ref{8} and \ref{81}. Thus $\eta_B$ is isomorphic.
Let $\underline f:B\rightarrow C$ be a morphism in $\uB$, then we can get the following diagram by Proposition \ref{8} and \ref{81}.
$$\xymatrix{
\sigma^+\sigma^-(B) \ar@{}[ddr]|{\circlearrowright} \ar@/^15pt/[rrr]^{\sigma^+\sigma^-(\underline f)} \ar[dd]_{\eta_B} &B^- \ar@{}[dr]|{\circlearrowright} \ar[l]^-{\beta} \ar[d]_{\underline {\gamma_B}} \ar[r]^{\sigma^-(\underline f)} &C^- \ar@{}[ddr]|{\circlearrowright} \ar[r]_-{\gamma} \ar[d]^{\underline {\gamma_C}} &\sigma^+\sigma^-(C) \ar[dd]^{\eta_C}\\
&B \ar@{}[dr]|{\circlearrowright} \ar[d]_{\underline {\alpha_B}} \ar[r]^{\underline f} &C \ar[d]^{\underline {\alpha_C}}\\
\sigma^-\sigma^+(B) \ar@/_15pt/[rrr]_{\sigma^-\sigma^+(\underline f)} \ar[r]^-{\alpha} &B^+ \ar[r]_{\sigma^+(\underline f)} &C^+ &\sigma^-\sigma^+(C) \ar[l]_-{\delta}
}
$$
Since
$$\delta(\sigma^-\sigma^+(\underline f))\eta_B\beta=(\sigma^+(\underline f))\underline {\alpha_B\gamma_B}=\underline {\alpha_C\gamma_C}(\sigma^-(\underline f))=\delta\eta_C(\sigma^+\sigma^-(\underline f))\beta$$
we get $(\sigma^-\sigma^+(\underline f))\eta_B=\eta_C(\sigma^+\sigma^-(\underline f))$ by Proposition \ref{8} and \ref{81}. Thus $\eta$ is a natural isomorphism.
\end{proof}

\section{Half exact functor}

By Proposition \ref{8.6}, we have an isomorphism of functors $\B \to \underline \h$
where $\pi:\B\rightarrow \uB$ denotes the canonical functor. Let $H:=\sigma^+ \circ \sigma^- \circ \pi \simeq \sigma^- \circ \sigma^+ \circ \pi:\B \to \underline \h$.
The aim of this section is to show the following theorem. %if $\mathcal P\subseteq \W$ and $\mathcal I\subseteq \W$.

\begin{thm}\label{8.11}
For any cotorsion pair $(\U,\V)$ in $\B$%satisfying $\mathcal P\subseteq \W$ and $\mathcal I\subseteq \W$
, the functor
$H:\B\rightarrow \underline \h$
is half exact.
\end{thm}

We call $H$ the associated \emph{half exact} functor to $(\U,\V)$.

\begin{prop}\label{8.7}
The functor $H$ has the following properties:
\begin{itemize}
\item[(a)] $H$ is an additive functor.
%\item[(a)] For any objects $A$ and $B$ in $\B$, $H(A\oplus B) \simeq H(A)\oplus H(B)$ in $\underline \h$.
\item[(b)] $H|_{\h}=\pi|_{\h}$.
\item[(c)] $H(\U)=0$ and $H(\V)=0$ hold. In particular, $H(\mathcal P)=0$ and $H(\mathcal I)=0$.
\item[(d)] For any reflection sequence $\xymatrix{B \ar@{ >->}[r]^{z} &Z \ar@{->>}[r] &U}$ for $B$, $H(z)$ is an isomorphism in $\underline \h$.
\item[(e)] For any coreflection sequence $\xymatrix{V \ar@{ >->}[r] &K \ar@{->>}[r]^{k} &B}$ for $B$, $H(k)$ is an isomorphism in $\underline \h$.
\end{itemize}
\end{prop}

\begin{proof}
(a) is followed by the definition of $H$ and Propositions \ref{8.2}, \ref{8.3} directly.

Since $\h=\B^+\cap \B^-$, by Proposition \ref{8.2}, \ref{8.3}, we get (b).

By Proposition \ref{8.2}, $\sigma^+(\uB^+)=0$, hence $H(\U)=0=H(\mathcal P)$ since $\mathcal P\subseteq \U$, dually we have $H(\V)=0=H(\mathcal I)$. Hence (c) holds.

For any reflection sequence, we have $H(z)=\sigma^-\circ\sigma^+(\underline z)=\sigma^-(\underline g)$ where $g:B^+\rightarrow Z$ is the morphism in the proof of Lemma \ref{8.5}. Since $\underline g$ is an isomorphism, we get $H(z)$ is an isomorphism in $\underline \h$. Thus (d) holds and by dual, (e) also holds.
\end{proof}

\begin{lem}\label{1}
Let $B$ be any object in $\B$, $\Hom_{\uB}(\U,\B^+)=0$ and $\Hom_{\uB}(B^-,\V)=0$ hold.
\end{lem}

\begin{proof}
We only show $\Hom_{\uB}(\U,\B^+)=0$, the other one is dual.\\
Since $B\in \B^+$, it admits a short exact sequence $V_B \rightarrowtail W_B\twoheadrightarrow B$
where $W_B\in \W$. Then any morphism from an object in $\U$ to $B$ factors through $W_B$, and the assertion follows.
\end{proof}

\begin{lem}\label{8.9}
%Suppose $\mathcal P\subseteq \W$.
Let
\begin{equation}\label{F3}
$$\quad \quad \quad \quad \quad \quad \quad \quad \quad \quad \quad \quad \quad \quad \quad \quad \quad \quad\xymatrix@C=0.6cm@R0.6cm{
\Omega U\ar@{ >->}[r]^q \ar[d]_f &P_U \ar@{->>}[r] \ar[d]^p &U \ar@{=}[d]\\
A \ar@{ >->}[r]_g &B \ar@{->>}[r]_h &U
}
$$
\end{equation}
be a commutative diagram satisfying $U\in \U$ and $P_U\in \mathcal P$. Then the sequence
$$H(\Omega U)\xrightarrow{H(f)} H(A)\xrightarrow{H(g)} H(B)\rightarrow 0$$
is exact in $\underline \h$.
\end{lem}

\begin{proof}
By Proposition \ref{7}, we get a commutative diagram by taking a pull-back of $g$ and $\gamma_B$
$$\xymatrix@C=0.6cm@R0.6cm{
V_0 \ar@{=}[r] \ar@{ >->}[d] &V_0 \ar@{ >->}[d]\\
L \ar@{ >->}[r]^{g'} \ar@{->>}[d]_l &B^- \ar@{->>}[d]^{\gamma_B} \ar@{->>}[r] &U \ar@{=}[d]\\
A \ar@{ >->}[r]_g &B \ar@{->>}[r]_h &U.
}
$$
By \cite[{Lemma 2.10}]{L}, $L\in \B^-$. We can obtain a commutative diagram of short exact sequences
$$\xymatrix@C=0.6cm@R0.6cm{
V_0 \ar@{=}[d] \ar@{ >->}[r] &L \ar[d] \ar@{->>}[r]^l &A \ar[d]^g\\
V_0 \ar@{=}[d] \ar@{ >->}[r] &B^- \ar[d] \ar@{->>}[r]^{\gamma_B} &B \ar[d]^j\\
V_0 \ar@{ >->}[r] &I^0 \ar@{->>}[r] &\Omega^-V_0
}
$$
where $j$ factors through $\V$ by Lemma \ref{8.5}, hence
$\xymatrix{V_0 \ar@{ >->}[r] &L \ar@{->>}[r]^{l} &A}$
is a coreflection sequence for $A$. By Proposition \ref{8.7}, $H(l)$ and $H(\gamma_B)$ are isomorphic in $\underline \h$. Thus, replacing $A$ by $L$ and $B$ by $B^-$, we may assume that $A,B\in \B^-$. Under this assumption, we show $H(g)$ is the cokernel of $H(f)$. We have $\Omega U\in \B^-$ by Lemma \ref{8.00}. For any $Q\in \h$, we have a commutative diagram
$$\xymatrix{
\Hom_{\uB}(H(B),Q) \ar[d]^{\simeq} \ar[rr]^-{\Hom_{\uB}(H(g),Q)} &&\Hom_{\uB}(H(A),Q) \ar[d]^{\simeq} \ar[rr]^-{\Hom_{\uB}(H(f),Q)} &&\Hom_{\uB}(H(\Omega U),Q) \ar[d]^{\simeq}\\
\Hom_{\uB}(\sigma^+(B),Q) \ar[d]^{\simeq} \ar[rr]^-{\Hom_{\uB}(\sigma^+(\underline g),Q)} &&\Hom_{\uB}(\sigma^+(A),Q) \ar[d]^{\simeq} \ar[rr]^-{\Hom_{\uB}(\sigma^+(\underline f),Q)} &&\Hom_{\uB}(\sigma^+(\Omega U),Q) \ar[d]^{\simeq}\\
\Hom_{\uB}(B,Q) \ar[rr]^-{\Hom_{\uB}(\underline g,Q)} &&\Hom_{\uB}(A,Q) \ar[rr]^-{\Hom_{\uB}(\underline f,Q)} &&\Hom_{\uB}(\Omega U,Q).
}
$$
So it suffices to show the following sequence
$$0\rightarrow \Hom_{\uB}(B,Q)\xrightarrow{\Hom_{\uB}(\underline g,Q)} \Hom_{\uB}(A,Q) \xrightarrow{\Hom_{\uB}(\underline f,Q)} \Hom_{\uB}(\Omega U,Q)$$
is exact.\\
We first show that $\Hom_{\uB}(\underline g,Q)$ is injective. Let $r:B\rightarrow Q$ be any morphism such that $\underline {rg}=0$. Take a commutative diagram of short exact sequences
$$\xymatrix@C=0.6cm@R0.6cm{
\Omega U^A \ar@{ >->}[r]^{q_A} \ar[d]_a &P_{U^A} \ar[d]^{p_A} \ar@{->>}[r] &U^A \ar@{=}[d]\\
A \ar@{ >->}[r]_{w^A} &W^A \ar@{->>}[r] &U^A.
}
$$
Since $rga$ factors through $\W$ and $\Ext^1_\B(W^A,\W)=0$, it factors through $q_A$. Thus there exists $c:W^A\rightarrow Q$ such that $cw^A=rg$.
$$\xymatrix{
\Omega U^A \ar[r]^{q_A} \ar[d]_a &P_{U^A} \ar[d]_{p_A} \ar@/^/[ddr]\\
A \ar[r]^{w^A} \ar@/_/[drr]_{rg} &W^A \ar@{.>}[dr]^c\\
&&Q
}
$$
As $\Ext^1_\B(U,W^A)=0$, there exists $d:B\rightarrow W^A$ such that $w^A=dg$. Hence $rg=cw^A=cdg$, then $r-cd$ factors through $U$.
$$\xymatrix{
A \ar[d]_{w^A} \ar@{ >->}[r]^g &B \ar@{->>}[rr] \ar[dl]^d_{\circlearrowright} \ar[d]^{r-cd} &&U \ar@{.>}@/^/[dll]_{\circlearrowright}\\
W^A \ar[r]_c &Q
}
$$
Since $\Hom_{\uB}(U,Q)=0$ by Lemma \ref{1}, we get that $\underline r=0$.\\
Assume $r':A\rightarrow Q$ satisfies $\underline {r'f}=0$, since $\Ext^1_\B(U,\W)=0$, $r'f$ factors through $q$. As the left square of (3) is a push-out, we get the following commutative diagram.
$$\xymatrix{
\Omega U \ar[r]^q \ar[d]_f &P_U \ar[d] \ar@/^/[ddr]\\
A \ar[r]^g \ar@/_/[drr]_{r'} &B \ar@{.>}[dr]\\
&&Q
}
$$
Hence $r'$ factors through $g$. This shows the exactness of
$$\Hom_{\uB}(B,Q)\xrightarrow{\Hom_{\uB}(\underline g,Q)} \Hom_{\uB}(A,Q) \xrightarrow{\Hom_{\uB}(\underline f,Q)} \Hom_{\uB}(\Omega U,Q).$$
\end{proof}

Dually, we have the following:
\begin{lem}\label{8.10}
%Suppose $\mathcal I\subseteq \W$.
Let
$$\xymatrix@C=0.6cm@R0.6cm{
V \ar@{ >->}[r]^f \ar@{=}[d] &A \ar@{->>}[r]^g \ar[d] &B \ar[d]^h\\
V \ar@{ >->}[r] &I^V \ar@{->>}[r] &\Omega^-V
}
$$
be a commutative diagram satisfying $V\in \V$ and $I^V\in \mathcal I$. Then the sequence
$$0\rightarrow H(A)\xrightarrow{H(g)} H(B)\xrightarrow{H(h)} H(\Omega^-V)$$
is exact in $\underline \h$.
\end{lem}

Now we are ready to prove Theorem \ref{8.11}.

\begin{proof}
Let $\xymatrix{A \ar@{ >->}[r]^{f} &B \ar@{->>}[r]^{g} &C}$
be any short exact sequence in $\B$. By Proposition \ref{PO}, we can get the following commutative diagram:
$$\xymatrix@C=0.6cm@R0.6cm{
\Omega U^A \ar@{ >->}[r]^b \ar[d]_a &P_{U^A} \ar[d] \ar@{->>}[r] &U^A \ar@{=}[d]\\
A \ar@{ >->}[r]^{v^A} \ar@{}[dr]|{PO} \ar@{ >->}[d]_f &V^A \ar@{->>}[r] \ar@{ >->}[d]^e &U^A \ar@{=}[d]\\
B \ar@{ >->}[r]_c \ar@{->>}[d]_g &D \ar@{->>}[r] \ar@{->>}[d]^d &U^A\\
C \ar@{=}[r] &C.
}
$$
From the first and second row from the top, we get an exact sequence $H(\Omega U^A) \xrightarrow{H(a)} H(A) \rightarrow 0$ by Lemma \ref{8.9}. From the first and the third row from the top, we get an exact sequence $H(\Omega U^A) \xrightarrow{H(fa)} H(B) \xrightarrow{H(c)} H(D) \rightarrow 0$ by Lemma \ref{8.9}. From the middle column, we get an exact sequence $0\rightarrow H(D)\xrightarrow{H(d)} H(C)$ by Lemma \ref{8.10}.
%By Proposition \ref{8.7}, Lemma \ref{8.9} and \ref{8.10}, we have exact sequences
%\begin{align*}
%H(\Omega U^A) \xrightarrow{H(a)} H(A) \rightarrow 0,\\
%H(\Omega U^A) \xrightarrow{H(fa)} H(B) \xrightarrow{H(c)} H(D) \rightarrow 0,\\
%0\rightarrow H(D)\xrightarrow{H(d)} H(C).
%\end{align*}
Now we can obtain an exact sequence $H(A)\xrightarrow{H(f)} H(B)\xrightarrow{H(g)} H(C)$.
\end{proof}

Now we prove the following general observation on half exact functors.

\begin{cor}\label{8.15}
Let $\mathcal A$ be an abelian category and $F:\B\rightarrow \mathcal A$ be a half exact functor satisfying $F(\mathcal P)=0$ and $F(\mathcal I)=0$. Then for any short exact sequence
$\xymatrix{A \ar@{ >->}[r]^{f} &B \ar@{->>}[r]^{g} &C}$
in $\B$, there exist morphisms $h: C\rightarrow \Omega^- A$ and $h': \Omega C\rightarrow A$ such that the following sequence
\begin{align*}
\cdots \xrightarrow{F(\Omega h')} F(\Omega A) \xrightarrow{F(\Omega f)} F(\Omega B) \xrightarrow{F(\Omega g)} F(\Omega C) \xrightarrow{F(h')} F(A) \xrightarrow{F(f)} F(B)\\
\xrightarrow{F(g)} F(C) \xrightarrow{F(h)} F(\Omega^- A) \xrightarrow{F(\Omega^- f)} F(\Omega^- B) \xrightarrow{F(\Omega^- g)} F(\Omega^- C) \xrightarrow{F(\Omega^- h)} \cdots
\end{align*}
is exact in $\mathcal A$. Moreover, If we have a commutative diagram of short exact sequences
$$\xymatrix{
A \ar@{ >->}[r]^{f} \ar[d]^a &B \ar@{->>}[r]^{g} \ar[d]^b &C \ar[d]^c\\
A_1 \ar@{ >->}[r]^{f_1} &B_1 \ar@{->>}[r]^{g_1} &C_1
}
$$
we can get a commutative diagram of long exact sequences in $\mathcal A$
$$\xymatrix{
\cdots \ar[r] &F(\Omega C) \ar[d]_-{F(\Omega c)} \ar[r]^-{F(h')} &F(A) \ar[r]^{F(f)} \ar[d]^{F(a)} &F(B) \ar[r]^{F(g)} \ar[d]^{F(b)} &F(C) \ar[d]^{F(c)} \ar[r]^-{F(h)}  &F(\Omega^- A) \ar[d]^{F(\Omega^- a)} \ar[r] &\cdots\\
\cdots \ar[r] &F(\Omega C_1)  \ar[r]_-{F(h_1')} &F(A_1) \ar[r]_{F(f_1)} &F(B_1) \ar[r]_{F(g_1)}  &F(C_1)  \ar[r]_-{F(h_1)}  &F(\Omega^- A_1) \ar[r] &\cdots
}
$$
\end{cor}

\begin{proof}
Since $F(\mathcal P)=0$ (resp. $F(\mathcal I)=0$), the functor $F$ can be regarded as a functor from $\B/\mathcal P$ (resp. $\B/\mathcal I$) to $\mathcal A$.\\
For convenience, we fix the following commutative diagram:
$$\xymatrix@C=0.6cm@R0.6cm{
\Omega A \ar@{ >->}[r]^{q_A} \ar[d]_{\Omega f} &P_A \ar@{->>}[r]^{p_A} \ar[d]^r &A \ar[d]^f\\
\Omega B \ar@{ >->}[r]^{q_B} \ar[d]_{\Omega g} &P_B \ar@{->>}[r]^{p_B} \ar[d]^k &B \ar[d]^g\\
\Omega C \ar@{ >->}[r]_{q_C} &P_C \ar@{->>}[r]_{p_C} &C.
}
$$
Since
$\xymatrix{A \ar@{ >->}[r]^{f} &B \ar@{->>}[r]^{g} &C}$
admits two commutative diagrams
%\begin{equation}\label{F4}
$$\xymatrix@C=0.6cm@R0.6cm{\Omega C \ar@{ >->}[r]^{q_C} \ar[d]_{h'} &P_C \ar@{->>}[r]^{p_C} \ar[d]^{l} &C \ar@{=}[d]\\
A \ar@{ >->}[r]_{f} &B \ar@{->>}[r]_{g} &C,
} \quad \xymatrix@C=0.6cm@R0.6cm{
A \ar@{ >->}[r]^{f} \ar@{=}[d] &B \ar@{->>}[r]^{g} \ar[d]^i &C \ar[d]^h\\
A \ar@{ >->}[r] &I^A \ar@{->>}[r]_{j} &\Omega^-A
}
$$
%\end{equation}
we get two short exact sequences by Proposition \ref{PO}:
$$\xymatrix{
\Omega C \ar@{ >->}[r]^-{\svecv{-q_C}{h'}} &P_C\oplus A \ar@{->>}[r]^-{\svech{l}{f}} &B,\\
} \quad \xymatrix{
B \ar@{ >->}[r]^-{\svecv{i}{g}} &I^A\oplus C \ar@{->>}[r]^-{\svech{-j}{h}} &\Omega^-A.
}
$$
%If $\mathcal P\subseteq \W$ and $\mathcal I\subseteq \W$,
They induce two exact sequences
$$\xymatrix{
F(\Omega C) \ar[r]^{F(h')} &F(A) \ar[r]^{F(f)} &F(B),\\
} \quad \xymatrix{
F(B) \ar[r]^{F(g)} &F(C) \ar[r]^-{F(h)} &F(\Omega^-A).
}
$$
by Theorem \ref{8.11}. Now it is enough to show that
\begin{itemize}
\item[(a)] $\xymatrix{A \ar@{ >->}[r]^{f} &B \ar@{->>}[r]^{g} &C}$ induces an exact sequence
$$F(\Omega A) \xrightarrow{F(\Omega f)} F(\Omega B) \xrightarrow{F(\Omega g)} F(\Omega C)\xrightarrow{F(h')} F(A).$$

\item[(b)] $\xymatrix{A \ar@{ >->}[r]^{f} &B \ar@{->>}[r]^{g} &C}$ induces an exact sequence
$$F(C) \xrightarrow{F(h)} F(\Omega^- A) \xrightarrow{F(\Omega^- f)} F(\Omega^- B) \xrightarrow{F(\Omega^- g)} F(\Omega^- C).$$
\end{itemize}
We only show the first one, the second is by dual.\\
The short exact sequence $\xymatrix{\Omega C \ar@{ >->}[r]^-{\svecv{-q_C}{h'}} &P_C\oplus A \ar@{->>}[r]^-{\svech{l}{f}} &B}$ admits the following commutative diagram
$$\xymatrix{\Omega B \ar@{ >->}[r]^{q_B} \ar[d]_x &P_B \ar@{->>}[r]^{p_B} \ar[d]^-{\svecv{k'}{m}} &B \ar@{=}[d]\\
\Omega C \ar@{ >->}[r]_-{\svecv{-q_C}{h'}} &P_C\oplus A \ar@{->>}[r]_-{\svech{l}{f}} &B
}
$$
which induces the following exact sequence
$$\xymatrix{\Omega B \ar@{ >->}[r]^-{\svecv{q_B}{-x}} &P_A\oplus \Omega C \ar@{->>}[rr]^-{\left(\begin{smallmatrix}
     k' &-q_C\\
     m &h'
    \end{smallmatrix}\right)} &&P_C\oplus A.}$$
%$$\xymatrix{
%H(\Omega B) \ar[r]^{H(x)} &H(\Omega C) \ar[r]^{H(h')} &H(A),\\
%}
%$$
%by Theorem \ref{8.11}.\\
We prove that $x+\Omega g$ factors through $\mathcal P$.\\
Since $fm+lk'=p_B\Rightarrow gfm+glk'=gp_B\Rightarrow p_Ck'=p_Ck$, there exists a morphism $n:P_B\rightarrow \Omega C$ such that $k-k'=q_Cn$. Thus we have $q_Cnq_B=kq_B-k'q_B=q_C\Omega g+q_Cx$, which implies that $x+\Omega g=nq_B$.\\
Hence we obtain an exact sequence $F(\Omega B) \xrightarrow{F(\Omega g)} F(\Omega C) \xrightarrow{F(h')} F(A).$\\
Since we have the following commutative diagram
$$\xymatrix{
\Omega A \ar@{ >->}[r]^{q_A} \ar[d]_{x'} &P_A \ar@{->>}[rr]^{p_A} \ar[d]^-{\svecv{s}{t}} &&A \ar[d]^-{\svecv{0}{1}}\\
\Omega B \ar@{ >->}[r]_-{\svecv{-q_B}{x}} &P_B\oplus \Omega C \ar@{->>}[rr]_-{\left(\begin{smallmatrix}
     k' &-q_C\\
     m &h'
    \end{smallmatrix}\right)} &&P_C\oplus A
}
$$
we can show that $x'+\Omega f$ factors through $\mathcal P$ using the same method.\\
Hence we get the following exact sequence
$$F(\Omega A) \xrightarrow{F(\Omega f)} F(\Omega B) \xrightarrow{F(\Omega g)} F(\Omega C) \xrightarrow{F(h')} F(A).$$
%We have
%$${\left(\begin{smallmatrix}
     %k' &-q_C\\
     %m &h'
    %\end{smallmatrix}\right)}{\left(\begin{smallmatrix}
     %s\\
     %t
    %\end{smallmatrix}\right)}={\left(\begin{smallmatrix}
     %k's-q_Ct\\
     %ms+h't
    %\end{smallmatrix}\right)}={\left(\begin{smallmatrix}
     %0\\
     %p_A
    %\end{smallmatrix}\right)}$$
%then $ms+h't=p_A\Rightarrow fms+fh't=fp_A\Rightarrow p_Bs-lk's+lq_Ct=p_Br\Rightarrow p_Bs$
Now we obtain a long exact sequence
\begin{align*}
\cdots \xrightarrow{F(\Omega h')} F(\Omega A) \xrightarrow{F(\Omega f)} F(\Omega B) \xrightarrow{F(\Omega g)} F(\Omega C) \xrightarrow{F(h')} F(A) \xrightarrow{F(f)} F(B)\\
\xrightarrow{F(g)} F(C) \xrightarrow{F(h)} F(\Omega^- A) \xrightarrow{F(\Omega^- f)} F(\Omega^- B) \xrightarrow{F(\Omega^- g)} F(\Omega^- C) \xrightarrow{F(\Omega^- h)} \cdots
\end{align*}
in $\underline \h$. Now let
$$\xymatrix@C=0.6cm@R0.6cm{
A \ar@{ >->}[r]^{f} \ar[d]^a &B \ar@{->>}[r]^{g} \ar[d]^b &C \ar[d]^c\\
A_1 \ar@{ >->}[r]^{f_1} &B_1 \ar@{->>}[r]^{g_1} &C_1
}
$$
a commutative diagram of short exact sequences, it is enough to check that we can get the following commutative diagram
$$\xymatrix@C=0.6cm@R0.6cm{
F(\Omega C) \ar[d]_-{F(\Omega c)} \ar[r]^-{F(h')} &F(A) \ar[r]^{F(f)} \ar[d]^{F(a)} &F(B) \ar[d]^{F(b)}\\
F(\Omega C_1)  \ar[r]_-{F(h_1')} &F(A_1) \ar[r]_{F(f_1)} &F(B_1)
}
$$
by applying $F$.\\
Since we have the following commutative diagrams
$$\xymatrix@C=0.6cm@R0.6cm{\Omega C \ar@{ >->}[r]^{q_C} \ar[d]_{h'} &P_C \ar@{->>}[r]^{p_C} \ar[d]^{l} &C \ar@{=}[d]\\
A \ar@{ >->}[r]_{f} &B \ar@{->>}[r]_{g} &C,
} \quad \xymatrix@C=0.6cm@R0.6cm{\Omega C_1 \ar@{ >->}[r]^{q_{C_1}} \ar[d]_{h_1'} &P_{C_1} \ar@{->>}[r]^{p_{C_1}} \ar[d]^{l_1} &C_1 \ar@{=}[d]\\
A_1 \ar@{ >->}[r]_{f_1} &B_1 \ar@{->>}[r]_{g_1} &C_1,
}
$$
They induce the following commutative diagram
$$\xymatrix{
\Omega C \ar@/_15pt/[dd]_{ah'} \ar@{.>}[d]^{\Omega c} \ar@{ >->}[r]^{q_C} &P_C \ar@{.>}[d]^{p} \ar@/_15pt/[dd]_(.75){bl} \ar@{->>}[r]^{p_C} &C \ar[d]^c\\
\Omega C_1 \ar@{ >->}[r] \ar[d]^{h_1'} &P_{C_1} \ar@{->>}[r]^{p_{C_1}} \ar[d]^{l_1} &C_1 \ar@{=}[d]\\
A_1 \ar@{ >->}[r]_{f_1} &B_1 \ar@{->>}[r]_{g_1} &C_1
}
$$
Now we get the following commutative diagram
$$\xymatrix{
\Omega C \ar[r]^-{\svecv{-q_C}{h'}} \ar[d]_{\Omega c} &P_C\oplus A \ar[r]^-{\svech{l}{f}} \ar[d]^{\left(\begin{smallmatrix}
     p &0\\
     0 &a
    \end{smallmatrix}\right)} &B \ar[d]^b\\
\Omega C_1 \ar[r]_-{\svecv{-q_{C_1}}{h_1'}} &P_{C_1}\oplus A_1 \ar[r]_-{\svech{l_1}{f_1}} &B_1
}
$$
which implies the following commutative diagram in $\mathcal A$ by applying $F$.
$$\xymatrix@C=0.6cm@R0.6cm{
F(\Omega C) \ar[d]_-{F(\Omega c)} \ar[r]^-{F(h')} &F(A) \ar[r]^{F(f)} \ar[d]^{F(a)} &F(B) \ar[d]^{F(b)}\\
F(\Omega C_1)  \ar[r]_-{F(h_1')} &F(A_1) \ar[r]_{F(f_1)} &F(B_1)
}
$$
\end{proof}

Since $H(\mathcal P)=H(\mathcal I)=0$, we can see from this proposition that $H$ has the property we claimed in the introduction.

For two subcategories $\B_1,\B_2\subseteq \B$, we denote $\add(\B_1*\B_2)$ by the subcategory which consists by the objects $X$ which admits a short exact sequence
$\xymatrix{B_1\ar@{ >->}[r] &X\oplus Y \ar@{->>}[r] &B_2}$
where $B_1\in \B_1$ and $B_2\in \B_2$.

%We show the following corollary which gives some information of the kernel of $H$.

\begin{prop}\label{8.14}
For any cotorsion pair $(\U,\V)$ on $\B$ %satisfying $\mathcal P\subseteq \W$, $\mathcal I\subseteq \W$
and any object $B\in \B$, the following are equivalent.
\begin{itemize}
\item[(a)] $H(B)=0$.

\item[(b)] $B\in \add(\U*\V)$. %For any object $\Omega U\in \Omega \U$, any morphism $u\in \Hom_\B(\Omega U, B)$ factors through $\U$.

%\item[(c)] For any object $\Omega^-V\in \Omega^-\V$, any morphism $v\in \Hom_\B(B, \Omega^-V)$ factors through $\V$.
\end{itemize}
\end{prop}

\begin{proof}
We first prove that (a) implies (b).\\
By Proposition \ref{8.3}, since $H(B)=\sigma^-\circ\sigma^+(B)=0$, we get that $B^+\in\V$, hence from commutative diagram \eqref{F1}
%$$\xymatrix@C=0.6cm@R0.6cm{
%V_B \ar@{=}[d] \ar@{ >->}[r] &U_B \ar@{ >->}[d]_{w'} \ar@{->>}[r]^{u_B} &B \ar@{ >->}[d]^{\alpha_B}\\
%V_B \ar@{ >->}[r]  &W^0 \ar@{->>}[r]_w \ar@{->>}[d] &B^+ \ar@{->>}[d]\\
%&U^0 \ar@{=}[r] &U^0
%}$$
we get a short exact sequence $\xymatrix{U_B \ar@{ >->}[r] &B\oplus W^0 \ar@{->>}[r] &B^+}$, which implies that $B\in \add(\U*\V)$.\\
We show that (b) implies (a).\\
This is followed by Theorem \ref{8.11} and Proposition \ref{8.7}.
\end{proof}

We denote $\add(\U*\V)$ by $\K$, which is called the \emph{kernel} of $H$.

\begin{rem}\label{ker}
Since $\Ext_\B(\U,\V)=0$, the subcategory $\V*\U=\{ U\oplus V \text{ } | \text{ } U\in \U,V\in \V \}$ is contained in $\U*\V$. Then $\K*\K=\add(\U*\V)*\add(\U*\V)=\add(\U*(\V*\U)*\V)\subseteq\add(\U*\U*\V*\V)=\K$, which implies that $\K$ is closed under extension.
\end{rem}

\begin{prop}\label{eqfunctor}
For any half exact functor $G:\B\rightarrow \underline \h$ such that $G(\K)=0$ and $G|_{\h}=\pi|_{\h}$, we get $G\simeq H$.
\end{prop}

To prove this proposition, we need the following lemma.

\begin{lem}\label{iso}
Let $G:\B\rightarrow \underline \h$ be a half exact functor such that $G(\K)=0$. Then
\begin{itemize}
\item[(a)] from the commutative diagram \eqref{F1}, we can get an isomorphism $G(\alpha_B)$.

\item[(b)] from the commutative diagram \eqref{F2}, we can get an isomorphism $G(\gamma_B)$.
\end{itemize}
\end{lem}

\begin{proof}
We only show (a), (b) is by dual\\
From the commutative diagram \eqref{F1}, we have an short exact sequence $\xymatrix{U_B\ar@{ >->}[r]^-{\svecv{-u_B}{w'}} &B\oplus W^0 \ar@{->>}[r]^-{\svech{\alpha_B}{w}} &B^+}$. Apply $G$ to this sequence, we get an exact sequence $0 \rightarrow G(B) \xrightarrow{G(\alpha_B)} G(B^+)$ in $\underline \h$. Apply $G$ again to the short exact sequence $\xymatrix{B \ar@{ >->}[r]^{\alpha_B} &B^+ \ar@{->>}[r] &U^0}$, we get an exact sequence $G(B) \xrightarrow{G(\alpha_B)} G(B^+) \rightarrow 0$ in $\underline \h$. Hence $G(\alpha_B)$ is both monomorphic and epimorphic in an abelian category $\underline \h$, which means $G(\alpha_B)$ is an isomorphism.
\end{proof}

Now we are ready to prove Proposition \ref{eqfunctor}.

\begin{proof}
%Since $\W\in \K$ and $G(\K)=0$, we have the following commutative diagram
%$$\xymatrix{
%\B \ar[rr]^-{\pi} \ar[dr]_G &&\uB \ar@{.>}[dl]^{\eta}\\
%&{\mathcal A}
%}
%$$
%Let $\iota: \underline \h \rightarrow \uB$ be the inclusion function, we have a functor $\eta\iota: \underline \h \rightarrow \mathcal A$ which makes the following diagram commute.
%$$\xymatrix{
%\B \ar[rr]^-H \ar[dr]_G &&{\underline \h} \ar@{.>}[dl]^{\eta\iota}\\
%&{\mathcal A}.
%}
%$$
%Now we assume $G|_{\h}=\pi|_{\h}$.
For any object $B\in \B$, we have the following morphisms $B \xrightarrow{\alpha_B} B^+  \xleftarrow{\gamma_{B^+}} (B^+)^-$ where $B^+\in \B^+$ and $(B^+)^-\in \h$. Since $G$ is half exact and $G(\K)=0$, by Lemma \ref{iso}, $G(\alpha_B)$ and $G(\gamma_{B^+})$ are isomorphisms in $\underline \h$. Since $G|_{\h}=\pi|_{\h}$, we obtain $G((B^+)^-)=(B^+)^-=H(B)$. Hence we get a morphism $\phi_B:G(B) \xrightarrow{G(\gamma_{B^+})^{-1}G(\alpha_B)} H(B)$. Let $f:X\rightarrow Y$ be a morphism in $\B$, we get the following commutative diagrams
$$\xymatrix{
X \ar[r]^{\alpha_X} \ar[d]_f &X^+ \ar[d]^{f^+}\\
Y \ar[r]_{\alpha_Y} &Y^+,\\
} \quad
\xymatrix{
(X^+)^- \ar[r]^{\gamma_{X^+}} \ar[d]_{(f^+)^-} &X^+ \ar[d]^{f^+}\\
(Y^+)^- \ar[r]_{\gamma_{Y^+}} &Y^+.}
$$
Since $\underline {(f^+)^-}=H(f)$, we get the following commutative diagram
$$\xymatrix{
G(X) \ar[r]^{G(\alpha_X)} \ar[d]_{G(f)} &G(X^+) \ar[d]^{G(f^+)} \ar[r]^-{G(\gamma_{X^+})^{-1}} &H(X) \ar[d]^{H(f)}\\
G(Y) \ar[r]_{G(\alpha_Y)} &G(Y^+) \ar[r]_-{G(\gamma_{Y^+})^{-1}} &H(Y)
}
$$
Now we can define a natural transformation $\phi:G\rightarrow H$ such that $\phi_B:G(B) \xrightarrow{G(\gamma_{B^+})^{-1}G(\alpha_B)} H(B)$, which is in fact an natural isomorphism. Hence $G\simeq H$.
\end{proof}

\section{Functors between different hearts}

%In this section we discuss the relationship between different hearts on $\B$ using the half exact functors we construct. First we fix some notations.
The half exact functor constructed in the previous section gives a useful way to study the relationship between the hearts of different cotorsion pairs on $\B$. First, we start with fixing some notations

Let $i\in \{ 1,2\}$. Let $(\U_i,\V_i)$ be a cotorsion pair on $\B$ and $\W_i=\U_i \cap \V_i$. Let $\B^+_i$ and $\B^-_i$ be the subcategories of $B$ defined in (1.1) and (1.2).
%\begin{itemize}
%\item[(a)] $\B^+_i$ is defined to be the full subcategory of $\B$, consisting of objects $B$ which admits a short exact sequence
%$$V_B\rightarrowtail U_B\twoheadrightarrow B$$
%where $U_B\in \W_i$ and $V_B\in \V_i$.

%\item[(b)] $\B^-_i$ is defined to be the full subcategory of $\B$, consisting of objects $B$ which admits a short exact sequence
%$$B\rightarrowtail V^B\twoheadrightarrow U^B$$
%where $V^B\in \W_i$ and $U^B\in \U_i$.
%\end{itemize}
Let $\h_i:=\B^+_i\cap\B^-_i$, then $\h_i/\W_i$ is the heart of $(\U_i,\V_i)$. Let $\pi_i:\B\rightarrow \B/\W_i$ be the canonical functor and $\iota_i: \h_i/\W_i \hookrightarrow \B/\W_i$ be the inclusion functor.

If $H_2(\W_1)=0$, which means $\W_1\subseteq \K_2$ by Proposition \ref{8.14}, then there exists a functor $h_{12}:\B/\W_1\rightarrow \h_2/\W_2$ such that $H_2=h_{12}\pi_1$.
$$\xymatrix@C=0.6cm@R0.6cm{
\B \ar[dr]_{H_2} \ar[rr]^{\pi_1} &&\B/\W_1 \ar@{.>}[dl]^{h_{12}}\\
&\h_2/\W_2
}$$
Hence we get a functor $\beta_{12}:=h_{12}\iota_1:\h_1/\W_1\rightarrow \h_2/\W_2$.

\begin{lem}\label{equi}
The following conditions are equivalent to each other.
\begin{itemize}
\item [(a)] $H_1(\U_2)=H_1(\V_2)=0$.
\item [(b)] $\K_2\subseteq\K_1$.
\end{itemize}
\end{lem}

\begin{proof}
By Proposition \ref{8.7} and Theorem \ref{8.11}, (b) implies (a).  Now we prove that (a) implies (b).\\
By Proposition \ref{8.14}, we get $\U_2\subseteq \K_1$ and $\V_2\subseteq \K_1$. Hence $\K_2=\add(\U_2*\V_2)\subseteq \add(\K_1*\K_1)=\add\K_1=\K_1$.
%Let $X\in \K_2$, then by definition, it admits a short exact sequence
%$$U_2\rightarrowtail X\oplus Y \twoheadrightarrow V_2$$
%where $U_2\in \U_2$ and $V_2 \in \V_2$. Since $U_2,V_2\in \K_1$, by definition, there exist two objects $A$ and $B$ such that $U_2\oplus A, V_2\oplus B\in \U_1*\V_1$. Thus we get a short exact sequence
%$$U_2\oplus A\rightarrowtail X\oplus Y\oplus A \oplus B \twoheadrightarrow V_2\oplus B.$$
%Hence by Remark \ref{ker}, $X\in \add((\U_1*\V_1)*(\U_1*\V_1))\subseteq \K_1$, which implies that $\K_2\subseteq\K_1$.
\end{proof}

\begin{prop}\label{nameless}
The functor $\beta_{12}$ is half exact. Moreover, if $\K_1\subseteq\K_2$, then $\beta_{12}$ is exact and $(\h_1\cap\K_2)/\W_1$ is a Serre subcategory of $\h_1/\W_1$.
\end{prop}

\begin{proof}
Let $0\rightarrow A \xrightarrow{\rho} B \xrightarrow{\mu} C\rightarrow 0$ be a short exact sequence in $\h_1/\W_1$, then $\mu$ admits a morphism $g:B\twoheadrightarrow C$ such that $\pi_1(g)=\mu$. We get the following commutative diagram
$$\xymatrix@C=0.6cm@R0.6cm{
V_C \ar@{ >->}[r] \ar@{=}[d] &{K_g}  \ar@{->>}[r]^{k_g} \ar[d]^a &B \ar[d]^g\\
V_C \ar@{ >->}[r] &{W_C} \ar@{->>}[r]_{w_C} &C}
$$
where $V_C\in \V_1$ and $W_C\in \W_1$. Then we obtain a short exact sequence
$$\xymatrix{K_g \ar@{ >->}[r]^-{\svecv{-a}{k_g}} &B\oplus W_C \ar@{->>}[r]^-{\svech{g}{w_C}} &C.}$$
By \cite[{Lemma 4.1}]{L}, $K_g\in \B^-_1$. By \cite[{Definition 3.8}]{L}, $K_g\in \B^+_1$. Hence $K_g\in \h_1$.Apply $H_1$ to the above diagram and short exact sequence, we get a short exact seqeunce $0\rightarrow K_g\xrightarrow{\underline{k_g}} B\xrightarrow{\mu} C\rightarrow 0$. Hence $K_g\simeq A$ in $\h_1/\W_1$. By Theorem \ref{8.11}, We get the an exact sequence
$$H_2(K_g)\xrightarrow{H_2(k_g)} H_2(B) \xrightarrow{H_2(g)} H_2(C)$$
which implies the following following exact sequence
$$\beta_{12}(A) \xrightarrow{\beta_{12}(\rho)} \beta_{12}(B) \xrightarrow{\beta_{12}(\mu)} \beta_{12}(C).$$
Hence $\beta_{12}$ is half exact. Now we prove that if $\K_1\subseteq\K_2$, which means $H_2(\U_1)=0=H_2(\V_1)$, then $\beta_{12}$ is exact.\\
In this case, we only need to show that $\beta_{12}(\rho)$ is a monomorphism and $\beta_{12}(\mu)$ is an epimorphism. We show that $\beta_{12}(\mu)$ is an epimorphism, the other part is by dual.\\
We have the following commutative diagram
$$\xymatrix@C=0.6cm@R0.6cm{
B \ar@{ >->}[r]^{w^B} \ar[d]_{g} &{W^B} \ar@{->>}[r] \ar[d]^b &{U^B} \ar@{=}[d]\\
C \ar@{ >->}[r]_{c_g} &{C_g} \ar@{->>}[r]_s &{U^B}}
$$
where $W_B\in \W_1$ and $U^B\in \U_1$. Since $\mu$ is epimorphism, by \cite[{Corollary 3.11}]{L}, $C_g\in \U_1$. Since we have the following short exact sequence
$\xymatrix{B \ar@{ >->}[r]^-{\svecv{g}{-h}} &C\oplus W^B \ar@{->>}[r]^-{\svech{c_g}{b}} &C_g.}$
By Theorem \ref{8.11}, We have an exact sequence $H_2(B) \xrightarrow{H_2(g)} H_2(C) \rightarrow 0$, which induces the following exact sequence
$\beta_{12}(B) \xrightarrow{\beta_{12}(\mu)} \beta_{12}(C)\rightarrow 0.$
Now we prove that $(\h_1\cap\K_2)/\W_1$ is a Serre subcategory of $\h_1/\W_1$.\\
Let $0\rightarrow A \xrightarrow{\rho} B \xrightarrow{\mu} C\rightarrow 0$ be a short exact sequence in $\h_1/\W_1$.\\
If $B\in (\h_1\cap\K_2)/\W_1$, since $\beta_{12}$ is exact and $\beta_{12}(B)=0$ by Proposition \ref{8.14}, we have $\beta_{12}(A)=0=\beta_{12}(C)$, which implies that $A,C\in (\h_1\cap\K_2)/\W_1$.\\
If $A,C\in (\h_1\cap\K_2)/\W_1$, since we have the following short exact sequence
$\xymatrix{K_g \ar@{ >->}[r]^-{\svecv{-a}{k_g}} &B\oplus W_C \ar@{->>}[r]^-{\svech{g}{w_C}} &C}$
in $\B$ such that $K_g\simeq A$ in $\h_1/\W_1$, we get that $B\in\add((\U_2*\V_2)*(\U_2*\V_2))\subseteq\K_2$. Hence $B\in (\h_1\cap\K_2)/\W_1$.
\end{proof}

We prove the following proposition, and we recall that a similar property has been proved for triangulated case in \cite[Lemma 6.3]{ZZ}.

\begin{prop}\label{61}
Let $(\U_1,\V_1)$ and $(\U_2,\V_2)$ be cotorsion pairs on $\B$. If $\W_1\subseteq \K_2\subseteq\K_1$, then we have a natural isomorphism $\beta_{21}\beta_{12}\simeq \id_{\h_1/\W_1}$ of functors.
\end{prop}

\begin{proof} 
Let $B\in \h_1$. By Definition \ref{re} and \ref{10}, we get the following commutative diagrams
$$\xymatrix@C=0.6cm@R0.6cm{
V_B \ar@{=}[d] \ar@{ >->}[r] &U_B \ar@{ >->}[d] \ar@{->>}[r] &B \ar@{ >->}[d]^{s_B}\\
V_B \ar@{ >->}[r]  &W^0 \ar@{->>}[r] \ar@{->>}[d] &B^+_2 \ar@{->>}[d]\\
&U^0 \ar@{=}[r] &U^0,\\
} \quad
\xymatrix@C=0.6cm@R0.6cm{
V_0 \ar@{ >->}[d] \ar@{=}[r] &{V_0} \ar@{ >->}[d]\\
(B^+_2)^-_2 \ar@{->>}[d]_{t_B} \ar@{ >->}[r] &{W_0} \ar@{->>}[d] \ar@{->>}[r] &{U^{B_2}} \ar@{=}[d]\\
B^+_2 \ar@{ >->}[r] &{V^{B_2}} \ar@{->>}[r] &{U^{B_2}}}
$$
where $U_B,U^0,U^{B_2}\in \U_2$, $V_B,V^{B_2},V_0\in \V_2$, $W_0,W^0\in \W_2$ and $(B^+_2)^-_2=H_2(B)$ in $\B/\W_2$. By Lemma \ref{equi}, we get $H_1(\U_2)=H_1(\V_2)=0$, by Lemma \ref{8.5} and Theorem \ref{8.11}, we get two isomorphisms $B\xrightarrow{H_1(s_B)} H_1(B^+_2)$ and $H_1((B^+_2)^-_2)\xrightarrow{H_1(t_B)} H_1(B^+_2)$ in $\h_1/\W_1$. Since $H_1((B^+_2)^-_2)=\beta_{21}\beta_{12}(B)$, we get a isomorphism $\rho_B:=H_1(t_B)^{-1}H_1(s_B):B\rightarrow \beta_{21}\beta_{12}(B)$ on $\h_1/\W_1$. Let $f:B\rightarrow C$ be a morphism in $\h_1$, we also denote it image in $\h_1/\W_1$ by $f$. By the definition of $H_2$, we get the following commutative diagrams in $\B$
$$\xymatrix{
B \ar[r]^{s_B} \ar[d]_f &B^+_2 \ar[d]^{f^+}\\
C \ar[r]_{s_C} &C^+_2,\\
} \quad
\xymatrix{
(B^+_2)^-_2 \ar[r]^{t_B} \ar[d]_{(f^+)^-} &B^+_2 \ar[d]^{f^+}\\
(C^+_2)^-_2 \ar[r]_{t_C} &C^+_2}
$$
where $\pi_2((f^+)^-)=H_2(f)$. Hence we obtain the following commutative diagram in $\h_1/\W_1$
$$\xymatrix{
B \ar[r]^-{\rho_B} \ar[d]_f &\beta_{21}\beta_{12}(B) \ar[d]^{\beta_{21}\beta_{12}(f)}\\
C \ar[r]_-{\rho_C} &\beta_{21}\beta_{12}(C)
}$$
which implies that $\beta_{21}\beta_{12}\simeq \id_{\h_1/\W_1}$.
\end{proof}

According to Proposition \ref{61}, we obtain the following corollary immediately.

\begin{cor}\label{suf}
If $\K_1=\K_2$, then we have an equivalence $\h_1/\W_1\simeq \h_2/\W_2$ between two hearts.
\end{cor}

For a cotorsion pair $(\U,\V)$, let $\mathcal C=\U\cap{^{\bot_1}}\U$, by defintion $\mathcal C$ is rigid. When $\U$ itself is rigid, we get $\mathcal C=\U$ and $\K=\V$. If $(\mathcal C,\K)$ is a cotorsion pair, denote its heart by $\underline \h_{\mathcal C}$, we have $\underline \h\simeq \underline \h_{\mathcal C}$ and denote $\beta:\underline \h\rightarrow \underline \h_{\mathcal C}$ by the equivalence. We have the following proposition.

\begin{prop}
Let $\mathcal C=\U\cap{^{\bot_1}}\U$, if $(\mathcal C,\K)$ is a cotorsion pair, let $H_{\mathcal C}$ be the associated half exact functor, then for a morphism $f:A\rightarrow B$ in $\B$, $H(f)=0$ if and only if $f$ factors through $\K$.
\end{prop}

\begin{proof}
By Proposition \ref{nameless} and \ref{eqfunctor} , $H_{\mathcal C}\simeq \beta H$. Thus $H(f)=0$ if and only if $H_{\mathcal C}(f)=0$. By Proposition \ref{8.3}, $H_{\mathcal C}(f)=0$ if and only if $f$ factors through $\K$.
\end{proof}

%Let $S=\{ \alpha \in \Mor(\h_2/\W_2) \text{ }|\text{ } \Ker(\alpha), \Coker(\alpha) \in (\h_2\cap\K_1)/\W_2\}$ and
Let $\overline \h_2$ be localization of $\h_2/\W_2$ with respect to $(\h_2\cap\K_1)/\W_2$, then $\overline \h_2$ is abelian. Since $\beta_{21}$ is exact and $\Ker (\beta_{21})=(\h_2\cap\K_1)/\W_2$, we get the following commutative diagram
$$\xymatrix@C=0.6cm@R0.6cm{
\h_2/\W_2 \ar[dr]_{L} \ar[rr]^{\beta_{21}} &&\h_1/\W_1\\
&\overline \h_2 \ar@{.>}[ur]_-{\overline {\beta_{21}}}
}$$
where $L$ is the localization functor which is exact and $\overline {\beta_{21}}$ is a faithful exact functor. Since $\overline {\beta_{21}}L\beta_{12}\simeq \id_{\h_1/\W_1}$, we get that $L\beta_{12}$ is fully-faithful. Now we prove that $L\beta_{12}$ is dense under the assumption of Proposition \ref{61}.\\

The following lemma is needed.

\begin{lem}\label{zero}
Let $(\U,\V)$ be a cotorsion pair. Then for any subcategory $\mathcal C\supseteq \K$ which is closed under extension, we have $H(\mathcal C)\subseteq (\h\cap \mathcal C)/\W$.
\end{lem}

\begin{proof}
Let $B\in \mathcal C$, we show that $H(B)\in (\h\cap \mathcal C)/\W$.
By Definition \ref{re} and \ref{10}, we get the following commutative diagrams
$$\xymatrix@C=0.6cm@R0.6cm{
V_B \ar@{=}[d] \ar@{ >->}[r] &U_B \ar@{ >->}[d] \ar@{->>}[r] &B \ar@{ >->}[d]\\
V_B \ar@{ >->}[r]  &W^0 \ar@{->>}[r] \ar@{->>}[d] &B^+ \ar@{->>}[d]\\
&U^0 \ar@{=}[r] &U^0,\\
} \quad
\xymatrix@C=0.6cm@R0.6cm{
V_0 \ar@{ >->}[d] \ar@{=}[r] &{V_0} \ar@{ >->}[d]\\
(B^+)^- \ar@{->>}[d] \ar@{ >->}[r] &{W_0} \ar@{->>}[d] \ar@{->>}[r] &{U^B} \ar@{=}[d]\\
B^+ \ar@{ >->}[r] &{V^B} \ar@{->>}[r] &{U^B}}
$$
where $U_B,U^0,U^B\in \U$, $V_B,V^B,V_0\in \V$, $W_0,W^0\in \W$ and $(B^+)^-=H(B)$ in $\uB$. From the left diagram we get $B^+\in \mathcal C*\U=\mathcal C$, since $\mathcal C\supseteq \K\supseteq \U$ and $\mathcal C$ is extension closed. Then from the right diagram we obtain $(B^+)^-\in \V*\mathcal C=\mathcal C$. Hence $H(B)\in (\h\cap \mathcal C)/\W$.
\end{proof}

Let $B\in \h_2$, by Definition \ref{re} and \ref{10}, we get the following commutative diagrams
$$\xymatrix@C=0.6cm@R0.6cm{
V_B \ar@{=}[d] \ar@{ >->}[r] &U_B \ar@{ >->}[d] \ar@{->>}[r] &B \ar@{ >->}[d]^{s_B}\\
V_B \ar@{ >->}[r]  &W^0 \ar@{->>}[r] \ar@{->>}[d] &B^+_1 \ar@{->>}[d]\\
&U^0 \ar@{=}[r] &U^0,\\
} \quad
\xymatrix@C=0.6cm@R0.6cm{
V_0 \ar@{ >->}[d] \ar@{=}[r] &{V_0} \ar@{ >->}[d]\\
(B^+_1)^-_1 \ar@{->>}[d]_{t_B} \ar@{ >->}[r] &{W_0} \ar@{->>}[d] \ar@{->>}[r] &{U^{B_1}} \ar@{=}[d]\\
B^+_1 \ar@{ >->}[r] &{V^{B_1}} \ar@{->>}[r] &{U^{B_1}}}
$$
where $U_B,U^0,U^{B_1}\in \U_1$, $V_B,V^{B_1},V_0\in \V_1$, $W_0,W^0\in \W_1$ and $(B^+_1)^-_1=H_1(B)$ in $\B/\W_1$. Since $H_2(\W_1)=0$, we get the following exact sequences by Theorem \ref{8.11}
$$H_2(U_B)\rightarrow B \rightarrow H_2(B^+_1) \rightarrow H_2(U^0),$$
$$H_2(V_0)\rightarrow H_2((B^+_1)^-_1) \rightarrow H_2(B^+_1) \rightarrow H_2(V^{B_1}).$$
By Lemma \ref{zero} $H_2(\U_1),H_2(\V_1)\subseteq (\h_2\cap\K_1)/\W_2$, apply $L$ to the above two exact sequences, since $L$ is exact, we get $B \simeq H_2(B^+_1) \simeq H_2((B^+_1)^-_1)=L\beta_{12}\beta_{21}(B)$ in $\overline \h_2$, which implies that $L\beta_{12}$ is dense.

Now we get the following theorem.

\begin{thm}\label{serre}
Let $(\U_1,\V_1)$, $(\U_2,\V_2)$ be cotorsion pairs on $\B$. If $\W_1\subseteq \K_2\subseteq\K_1$, then we have an equivalence $L\beta_{12}: \h_1/\W_1 \rightarrow \overline \h_2 $.
\end{thm}

In the rest of this section, we discuss about the relationship between the heart of a twin cotorsion pair and the hearts of its two components.

First we recall the definition of the twin cotorsion pair. A pair of cotorsion pairs $(\U_1,\V_1)$, $(\U_2,\V_2)$ is called a twin cotorsion pair if $\U_1\subseteq \U_2$. This condition is equivalent to $\V_2\subseteq \V_1$ and also equivalent to $\Ext^1_\B(\U_1,\V_2)=0$. We introduce some notations.

Let $\W_t:=\V_1\cap \U_2$.
\begin{itemize}
\item[(a)] $\B^+_t$ is defined to be the full subcategory of $\B$, consisting of objects $B$ which admits a short exact sequence
$V_B\rightarrowtail U_B\twoheadrightarrow B$
where $U_B\in \W_t$ and $V_B\in \V_2$.

\item[(b)] $\B^-_t$ is defined to be the full subcategory of $\B$, consisting of objects $B$ which admits a short exact sequence
$B\rightarrowtail V^B\twoheadrightarrow U^B$
where $V^B\in \W_t$ and $U^B\in \U_1$.
\end{itemize}
Denote $\h_t:=\B^+_t\cap\B^-_t$, $\h_t/\W_t$ is called the \emph{heart} of $(\U_1,\V_1),(\U_2,\V_2)$.

\begin{prop}\label{63}
Let $(\U_1,\V_1),(\U_2,\V_2)$ be a twin cotorsion pair on $\B$ and $f:A\rightarrow B$ be a morphism in $\h_t$, then $H_k(f)=0$ $(k=1 \text{ or } 2)$ if and only if $f$ factors through $\W_t$.
\end{prop}

\begin{proof}
We only prove the case $k=2$, the other case is by dual.\\
The "if" is followed directly by Proposition \ref{8.2}. Now we prove the "only if" part.\\
Since $H_k(f)=0$, by Proposition \ref{8.3} and \ref{8.6}, we get in the following commutative diagram
$$\xymatrix@C0.6cm@R0.6cm{
A \ar@{ >->}[d]_{\alpha_A} \ar[r]^f &B \ar@{ >->}[d]^{\alpha_B} &U_B \ar@{ >->}[d]^{w'} \ar@{->>}[l]_{u_B} &V_B \ar@{=}[d] \ar@{ >->}[l]\\
A^+ \ar@{->>}[d] \ar[r]_{f^+} &B^+ \ar@{->>}[d] &W^0 \ar@{->>}[d] \ar@{->>}[l]^w &V_B \ar@{ >->}[l]\\
U^0_A \ar[r] &U^0 \ar@{=}[r] &U^0
}$$
which is similar as in Proposition \ref{8.2}, where $U^0_A,U^0\in \U_2$, $V_B\in \V_2$, $U_B\in \W_t$ and $W^0\in \W_2$, $f^+$ factors through an object $V\in \V_2$. Since $A,B\in \h_t$, by \cite[Lemma 2.10]{L}, $A^+,B^+\in \B^-_t$. Hence there exits a diagram
$$\xymatrix@C0.6cm@R0.6cm{
A^+ \ar[dr]^a \ar[dd]_{f^+} \ar@{ >->}[rr]^-{w^A} &&W^A \ar@{->>}[r] &U^A\\
&V \ar[dl]^b\\
B^+ \ar@{ >->}[rr] &&W^B \ar@{->>}[r] &U^B
}\
$$
where $W^A,W^B\in \W_t$ and $U^A,U^B\in \U_1$. Since $\Ext^1_\B(U^A,V)=0$, there exists a morphism $c:W^A\rightarrow V$ such that $f^+=bcw^A$. Now using the same argument as in Proposition \ref{8.2}, we get that $f$ factors through $U_B\in \W_t$.
\end{proof}

Let $\pi_t:\B\rightarrow \B/\W_t$ be the canonical functor and $\iota_t: \h_t/\W_t \hookrightarrow \B/\W_t$ be the inclusion functor.

Let $k \in \{ 1,2 \}$, since $H_k(W_t)=0$ by Proposition \ref{8.7}, there exists a functor $h_k:\B/\W_t\rightarrow \h_k/\W_k$ such that $H_k=h_k\pi_t$.
$$\xymatrix@C0.6cm@R0.6cm{
\B \ar[dr]_{H_k} \ar[rr]^{\pi_t} &&\B/\W_t \ar@{.>}[dl]^{h_k}\\
&\h_k/\W_k
}$$
Hence we get a functor $\beta_k:=h_k\iota_t:\h_t/\W_t\rightarrow \h_k/\W_k$ and the following corollary.

\begin{cor}
Let $(\U_1,\V_1),(\U_2,\V_2)$ be a twin cotorsion pair on $\B$, then $\beta_k:\h_t/\W_t\rightarrow \h_k/\W_k$ $(k \in \{ 1,2 \})$ is faithful.
\end{cor}

This corollary also implies that if $\h_1/\W_1=0$ or $\h_2/\W_2=0$, $\h_t/\W_t$ is also zero.

%According to this proposition, we can get the following corollary.

%\begin{cor}\label{64}
%Let $(\U_1,\V_1),(\U_2,\V_2)$ be a twin cotorsion pair on $\B$, then $\h_t/\W_t=0$ if $\h_1/\W_1=0$ or $\h_2/\W_2=0$.
%\end{cor}

Moreover, we have the following proposition.

\begin{prop}
Let $(\U_1,\V_1),(\U_2,\V_2)$ be a twin cotorsion pair on $\B$. If $\h_t/\W_t=0$, then $\h_1\subseteq \U_2$ and $\h_2\subseteq \V_1$.
%\begin{itemize}
%\item[(a)] If $\h_t/\W_t=0$, then $\h_1\subseteq \U_2$. In particular, $\h_2/\W_2=0$ implies that $\h_1\subseteq \U_2$.

%\item[(b)] If $\h_t/\W_t=0$, then $\h_2\subseteq \V_1$. In particular, $\h_1/\W_1=0$ implies that $\h_2\subseteq \V_1$.
%\end{itemize}
\end{prop}

\begin{proof}
We only prove that $\h_t/\W_t=0$ implies $\h_1\subseteq \U_2$, the other one is by dual.\\
Let $B\in \h_1$, since $B^-_1\subseteq B^-_t$ by definition, in the following diagram
$$\xymatrix@C0.6cm@R0.6cm{
V_B \ar@{=}[d] \ar@{ >->}[r] &U_B \ar@{ >->}[d] \ar@{->>}[r] &B \ar@{ >->}[d]\\
V_B \ar@{ >->}[r]  &W^0 \ar@{->>}[r] \ar@{->>}[d] &B^+ \ar@{->>}[d]\\
&U^0 \ar@{=}[r] &U^0
}$$
where $U_B\in \U_2$, $V_B\in \V_2$, $U^0\in \U_1$ and $W^0\in \W_t$, we get $B^+\in \h_t$ by \cite[Lemma 2.10]{L}. If $\h_t/\W_t=0$, then $B^+\in \W_t$. By \cite[{Lemma 3.4}]{L}, $B\in \U_2$.
\end{proof}

%This proposition also implies that if $\h_1-\U_2 \neq \emptyset$ or $\h_2-\V_1 \neq \emptyset$, then $\h_T/\W_T \neq 0$.\\

\section{Examples}

%The aim of this section is to describe the associated functor $H$ in a simple way for some special cases.
\begin{exm}
Let $\Lambda$ be the $k$-algebra given by the quiver
$$
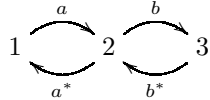
\xymatrix{
1 \ar@/^10pt/[r]^a &2 \ar@/^10pt/[l]^{a^*} \ar@/^10pt/[r]^b &3 \ar@/^10pt/[l]^{b^*}}$$
and bounded by the relations $a^*a=0=bb^*$, $aa^*=b^*b$. The AR-quiver of $\B=\mod\Lambda$ is given by
$$
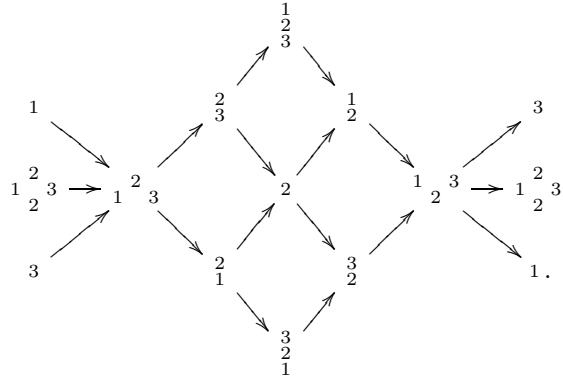
\xymatrix@C=0.4cm@R0.4cm{
&&&{\begin{smallmatrix}
1\\
2\\
3
\end{smallmatrix}} \ar[dr]\\
{\begin{smallmatrix}
1
\end{smallmatrix}} \ar[dr]
&&{\begin{smallmatrix}
2\\
3
\end{smallmatrix}} \ar[ur] \ar[dr]
&&{\begin{smallmatrix}
1\\
2
\end{smallmatrix}} \ar[dr] &&{\begin{smallmatrix}
3
\end{smallmatrix}}\\
{\begin{smallmatrix}
&2&\\
1&&3\\
&2&
\end{smallmatrix}} \ar[r]
&{\begin{smallmatrix}
&2&\\
1&&3
\end{smallmatrix}} \ar[dr] \ar[ur]
&&{\begin{smallmatrix}
2
\end{smallmatrix}} \ar[dr] \ar[ur]
&&{\begin{smallmatrix}
1&&3\\
&2&
\end{smallmatrix}} \ar[dr] \ar[ur] \ar[r] &{\begin{smallmatrix}
&2&\\
1&&3\\
&2&
\end{smallmatrix}}\\
{\begin{smallmatrix}
3
\end{smallmatrix}} \ar[ur]
&&{\begin{smallmatrix}
2\\
1
\end{smallmatrix}} \ar[ur] \ar[dr]
&&{\begin{smallmatrix}
3\\
2
\end{smallmatrix}} \ar[ur] &&{\begin{smallmatrix}
1
\end{smallmatrix}.}\\
&&&{\begin{smallmatrix}
3\\
2\\
1
\end{smallmatrix}} \ar[ur]
}$$
We denote by "$\circ$" in the AR-quiver the indecomposable objects belong to a subcategory and by "$\cdot$" the indecomposable objects do not. \\
Let $\U_1$ and $\V_1$ be the full subcategories of $\mod \Lambda$ given by the following diagram.
$$
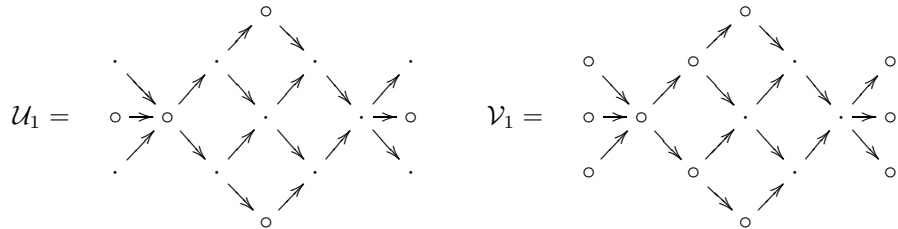
\xymatrix@C=0.3cm@R0.3cm{
&&&&\circ \ar[dr]\\
&\cdot \ar[dr]
&&\cdot \ar[ur] \ar[dr]
&&\cdot \ar[dr] &&\cdot\\
{\U_1=}&\circ \ar[r]
&\circ \ar[dr] \ar[ur]
&&\cdot \ar[dr] \ar[ur]
&&\cdot \ar[dr] \ar[ur] \ar[r]
&\circ\\
&\cdot \ar[ur]
&&\cdot \ar[ur] \ar[dr]
&&\cdot \ar[ur]
&&\cdot\\
&&&&\circ \ar[ur]\\} \quad \quad
\xymatrix@C=0.3cm@R0.3cm{
&&&&\circ \ar[dr]\\
&\circ \ar[dr]
&&\circ \ar[ur] \ar[dr]
&&\cdot \ar[dr] &&\circ\\
{\V_1=}&\circ \ar[r]
&\circ \ar[dr] \ar[ur]
&&\cdot \ar[dr] \ar[ur]
&&\cdot \ar[dr] \ar[ur] \ar[r]
&\circ\\
&\circ \ar[ur]
&&\circ \ar[ur] \ar[dr]
&&\cdot \ar[ur]
&&\circ\\
&&&&\circ \ar[ur]
}$$
The heart $\h_1/\W_1=\add({\begin{smallmatrix}
2
\end{smallmatrix}})$ and $\h_1\simeq \mod (\U_1/\mathcal P)$ by \cite[{Theorem 3.2}]{DL}. Now let $\U_2$ and $\V_2$ be the full subcategories of $\mod \Lambda$ given by the following diagram.
$$
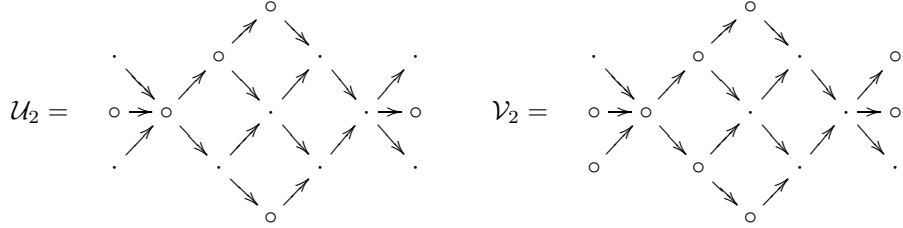
\xymatrix@C=0.3cm@R0.3cm{
&&&&\circ \ar[dr]\\
&\cdot \ar[dr]
&&\circ \ar[ur] \ar[dr]
&&\cdot \ar[dr] &&\cdot\\
{\U_2=}&\circ \ar[r]
&\circ \ar[dr] \ar[ur]
&&\cdot \ar[dr] \ar[ur]
&&\cdot \ar[dr] \ar[ur] \ar[r]
&\circ\\
&\cdot \ar[ur]
&&\cdot \ar[ur] \ar[dr]
&&\cdot \ar[ur]
&&\cdot\\
&&&&\circ \ar[ur]\\} \quad \quad
\xymatrix@C=0.3cm@R0.3cm{
&&&&\circ \ar[dr]\\
&\cdot \ar[dr]
&&\circ \ar[ur] \ar[dr]
&&\cdot \ar[dr] &&\circ\\
{\V_2=}&\circ \ar[r]
&\circ \ar[dr] \ar[ur]
&&\cdot \ar[dr] \ar[ur]
&&\cdot \ar[dr] \ar[ur] \ar[r]
&\circ\\
&\circ \ar[ur]
&&\circ \ar[ur] \ar[dr]
&&\cdot \ar[ur]
&&\cdot\\
&&&&\circ \ar[ur]
}$$
The heart $\h_2/\W_2=\add({\begin{smallmatrix}
1
\end{smallmatrix}},{\begin{smallmatrix}
2
\end{smallmatrix}})$. Since $\W_1=\U_1\subseteq \U_2\subseteq \V_2 \subseteq \V_1$, by Theorem \ref{serre}, $\overline \h_2\simeq \h_1/\W_1$.\\
Moreover, $\V_1/\U_1$ has a triangulated category structure, and $(\U_2/\U_1,\V_2/\U_1)$ is a cotorsion pair on it. The Serre subcategory $(\h_2\cap\K_1)/\W_2=\add({\begin{smallmatrix}
1
\end{smallmatrix}})$ is the heart of $(\U_2/\U_1,\V_2/\U_1)$.
%$$\xymatrix{
%{\begin{smallmatrix}
%2
%\end{smallmatrix}}\oplus {\begin{smallmatrix}
%3\\
%2\\
%1
%\end{smallmatrix}} \ar@{ >->}[r] &{\begin{smallmatrix}
%3\\
%2\\
%1
%\end{smallmatrix}}\oplus {\begin{smallmatrix}
%1\\
%3\\
%2
%\end{smallmatrix}} \ar@{->>}[r] &{\begin{smallmatrix}
%1\\
%3
%\end{smallmatrix}}
%}$$
%from the AR-quiver we can find that any morphism ends at ${\begin{smallmatrix}
%3\\
%2
%\end{smallmatrix}}$ factors through $\U$, hence by Corollary \ref{8.14}, we have $H({\begin{smallmatrix}
%3\\
%2
%\end{smallmatrix}})=0$.
\end{exm}

Recall that a subcategory $\M$ of $\B$ is called rigid if $\Ext^1_\B(\M,\M)=0$, $\M$ is cluster tilting if it satisfies
\begin{itemize}
\item [(a)] $\M$ is contravariantly finite and covariantly finite in $\B$.
\item [(b)] $X\in \M$ if and only if $\Ext^1_{\B}(X,\M)=0$.
\item [(c)] $X\in \M$ if and only if $\Ext^1_{\B}(\M,X)=0$.
\end{itemize}

If $\M$ is a cluster tilting subcategory of $\B$, then $(\M,\M)$ is a cotorsion pair on $\B$ (see \cite[{Proposition 10.5}]{L}).. In this case we have $\h=\B^-=\B^+=\B$, $\sigma^-=\sigma^+=\text{id}$ and $H=\pi$.

%Let $\Lambda$ be the $k$-algebra by the quiver
%$$\xymatrix{1 &2 \ar[l] &\cdot\cdot\cdot \ar[l] &n \ar[l]}$$
%with the relations such that all the path of length 2 equals to zero. Let $n>1$. The AR-quiver of $\Lambda$ is given by
%$$\xymatrix@C=0.4cm@R0.4cm{
%{\begin{smallmatrix}
%1
%\end{smallmatrix}} \ar[dr]
%&&{\begin{smallmatrix}
%2
%\end{smallmatrix}} \ar[dr]
%&&\cdot\cdot\cdot  &\cdot\cdot\cdot \ar[dr]
%&&{\begin{smallmatrix}
%n
%\end{smallmatrix}}\\
%&{\begin{smallmatrix}
%2&&\\
%&1&
%\end{smallmatrix}} \ar[ur]
%&&{\begin{smallmatrix}
%3&&\\
%&2&
%\end{smallmatrix}} \ar[ur]
%&&&{\begin{smallmatrix}
%n&&\\
%&n-1&
%\end{smallmatrix}} \ar[ur]
%}
%$$
\begin{exm}
Let $\Lambda$ be the $k$-algebra given by the quiver
$$
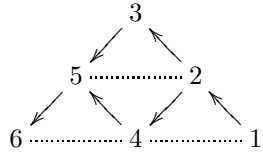
\xymatrix@C=0.4cm@R0.4cm{
&&3 \ar[dl]\\
&5 \ar[dl] \ar@{.}[rr] &&2 \ar[dl] \ar[ul]\\
6 \ar@{.}[rr] &&4 \ar[ul] \ar@{.}[rr] &&1 \ar[ul]}$$
with mesh relations. The AR-quiver of $\B:=\mod\Lambda$ is given by
$$
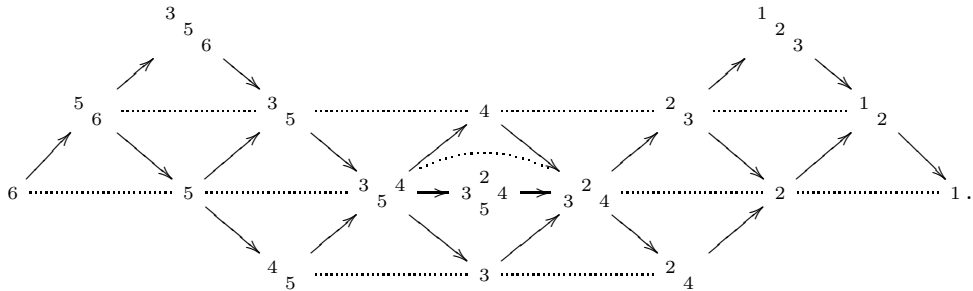
\xymatrix@C=0.4cm@R0.4cm{
&&{\begin{smallmatrix}
3&&\\
&5&\\
&&6
\end{smallmatrix}} \ar[dr] &&&&&&{\begin{smallmatrix}
1&&\\
&2&\\
&&3
\end{smallmatrix}} \ar[dr]\\
&{\begin{smallmatrix}
5&&\\
&6&
\end{smallmatrix}} \ar[ur] \ar@{.}[rr] \ar[dr] &&{\begin{smallmatrix}
3&&\\
&5&
\end{smallmatrix}} \ar@{.}[rr] \ar[dr] &&{\begin{smallmatrix}
4
\end{smallmatrix}} \ar@{.}[rr] \ar[dr] &&{\begin{smallmatrix}
2&&\\
&3&
\end{smallmatrix}} \ar[ur] \ar@{.}[rr] \ar[dr] &&{\begin{smallmatrix}
1&&\\
&2&
\end{smallmatrix}} \ar[dr]\\
{\begin{smallmatrix}
6
\end{smallmatrix}} \ar[ur] \ar@{.}[rr] &&{\begin{smallmatrix}
5
\end{smallmatrix}} \ar[ur] \ar@{.}[rr] \ar[dr] &&{\begin{smallmatrix}
3&&4\\
&5&
\end{smallmatrix}} \ar[ur] \ar[r] \ar[dr] \ar@{.}@/^15pt/[rr] &{\begin{smallmatrix}
&2&\\
3&&4\\
&5&
\end{smallmatrix}} \ar[r] &{\begin{smallmatrix}
&2&\\
3&&4
\end{smallmatrix}} \ar[ur] \ar@{.}[rr] \ar[dr] &&{\begin{smallmatrix}
2
\end{smallmatrix}} \ar[ur] \ar@{.}[rr] &&{\begin{smallmatrix}
1
\end{smallmatrix}}.\\
&&&{\begin{smallmatrix}
4&&\\
&5&
\end{smallmatrix}} \ar[ur] \ar@{.}[rr] &&{\begin{smallmatrix}
3
\end{smallmatrix}} \ar[ur] \ar@{.}[rr] &&{\begin{smallmatrix}
2&&\\
&4&
\end{smallmatrix}} \ar[ur]
}$$
Let $\U_1$ and $\V_1$ be the full subcategories of $\mod \Lambda$ given by the following diagram.
$$
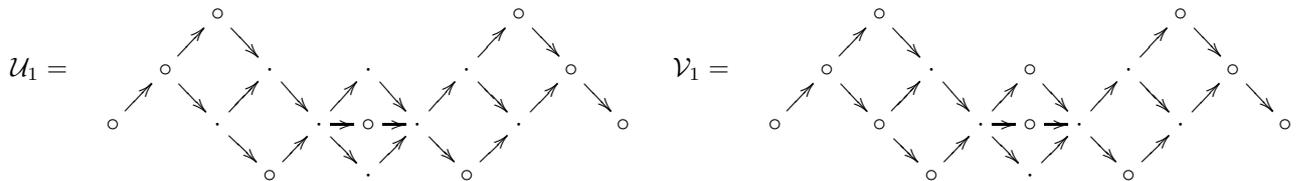
\xymatrix@C=0.3cm@R0.3cm{
&&&\circ \ar[dr] &&&&&&\circ \ar[dr]\\
{\U_1=} &&\circ \ar[ur]  \ar[dr] &&\cdot  \ar[dr] &&\cdot  \ar[dr] &&\cdot  \ar[ur]  \ar[dr] &&\circ \ar[dr]\\
&\circ \ar[ur]  &&\cdot \ar[ur]  \ar[dr] &&\cdot \ar[ur] \ar[r] \ar[dr] &\circ \ar[r] &\cdot \ar[ur] \ar[dr] &&\cdot \ar[ur] &&\circ\\
&&&&\circ \ar[ur] &&\cdot \ar[ur] &&\circ \ar[ur]
\\} \quad
\xymatrix@C=0.3cm@R0.3cm{
&&&\circ \ar[dr] &&&&&&\circ \ar[dr]\\
{\V_1=} &&\circ \ar[ur]  \ar[dr] &&\cdot  \ar[dr] &&\circ  \ar[dr] &&\cdot  \ar[ur]  \ar[dr] &&\circ \ar[dr]\\
&\circ \ar[ur]  &&\circ \ar[ur]  \ar[dr] &&\cdot \ar[ur] \ar[r] \ar[dr] &\circ \ar[r] &\cdot \ar[ur] \ar[dr] &&\cdot \ar[ur] &&\circ\\
&&&&\circ \ar[ur] &&\cdot \ar[ur] &&\circ \ar[ur]
}$$
Then $(\U_1,\V_1)$ is a cotorsion pair on $\mod \Lambda$. The heart $\h_1/\W_1$ is the following.
$$
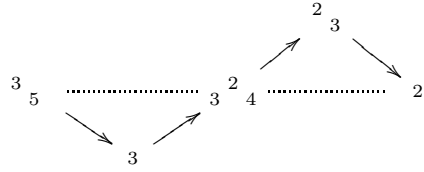
\xymatrix@C=0.4cm@R0.4cm{
&&&{\begin{smallmatrix}
2&\ \\
&3
\end{smallmatrix}}\ar[dr]\\
{\begin{smallmatrix}
3&&\ \\
&5&
\end{smallmatrix}}\ar[dr] \ar@{.}[rr]
&&{\begin{smallmatrix}
&2&\ \\
3&&4
\end{smallmatrix}}\ar[ur] \ar@{.}[rr]
&&{\begin{smallmatrix}
\ &2&\
\end{smallmatrix}}\\
&{\begin{smallmatrix}
\ &3&\
\end{smallmatrix}} \ar[ur]}$$
The only indecomposable object which does not lie in $\h_1$ or $\U_1$,$\V_1$ is ${\begin{smallmatrix}
3&&4\\
&5&
\end{smallmatrix}}$, since we have the following commutative diagram
$$
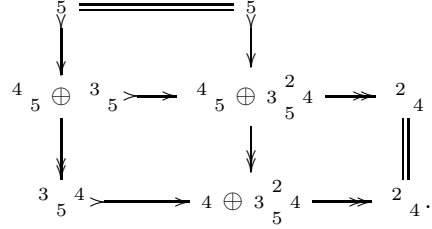
\xymatrix@C0.6cm@R0.6cm{
{\begin{smallmatrix}
5
\end{smallmatrix}} \ar@{=}[r] \ar@{ >->}[d] &{\begin{smallmatrix}
5
\end{smallmatrix}} \ar@{ >->}[d]\\
{\begin{smallmatrix}
&4&\\
&&5
\end{smallmatrix}}\oplus{\begin{smallmatrix}
&3&\\
&&5
\end{smallmatrix}} \ar@{ >->}[r] \ar@{->>}[d] &{\begin{smallmatrix}
&4&\\
&&5
\end{smallmatrix}}\oplus{\begin{smallmatrix}
&2&\\
3&&4\\
&5&
\end{smallmatrix}} \ar@{->>}[r] \ar@{->>}[d] &{\begin{smallmatrix}
&2&\\
&&4
\end{smallmatrix}} \ar@{=}[d]\\
{\begin{smallmatrix}
3&&4\\
&5&
\end{smallmatrix}} \ar@{ >->}[r] &{\begin{smallmatrix}
4
\end{smallmatrix}}\oplus{\begin{smallmatrix}
&2&\\
3&&4\\
&5&
\end{smallmatrix}} \ar@{->>}[r] &{\begin{smallmatrix}
&2&\\
&&4\\
\end{smallmatrix}}.
}$$
We get $H_1({\begin{smallmatrix}
3&&4\\
&5&
\end{smallmatrix}})={\begin{smallmatrix}
3&&\\
&5&
\end{smallmatrix}}$ since ${\begin{smallmatrix}
&4&\\
&&5
\end{smallmatrix}}\in \mathcal P$.
Let
$$
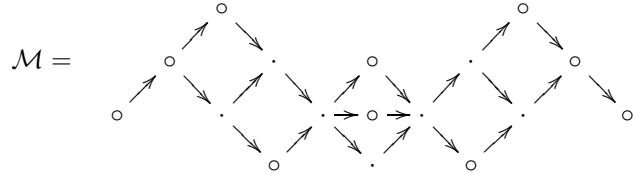
\xymatrix@C=0.3cm@R0.3cm{
&&&\circ \ar[dr] &&&&&&\circ \ar[dr]\\
{\M=} &&\circ \ar[ur]  \ar[dr] &&\cdot  \ar[dr] &&\circ  \ar[dr] &&\cdot  \ar[ur]  \ar[dr] &&\circ \ar[dr]\\
&\circ \ar[ur]  &&\cdot \ar[ur]  \ar[dr] &&\cdot \ar[ur] \ar[r] \ar[dr] &\circ \ar[r] &\cdot \ar[ur] \ar[dr] &&\cdot \ar[ur] &&\circ\\
&&&&\circ \ar[ur] &&\cdot \ar[ur] &&\circ \ar[ur]
}$$
Since $\M$ is a cluster tilting subcategory of $\B$, $(\U_2,\V_2)=(\M,\M)$ is a cotorsion pair. The heart $\h_2/\W_2=\mod \Lambda/\M$ is the following.
$$
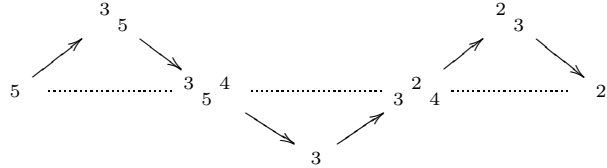
\xymatrix@C=0.4cm@R0.4cm{
&{\begin{smallmatrix}
&3&\ \\
&&5
\end{smallmatrix}}\ar[dr]
&&&&{\begin{smallmatrix}
2&\ \\
&3
\end{smallmatrix}}\ar[dr]\\
{\begin{smallmatrix}
\ &5&\
\end{smallmatrix}} \ar[ur] \ar@{.}[rr]
&&{\begin{smallmatrix}
3&&4\ \\
&5&
\end{smallmatrix}}\ar[dr] \ar@{.}[rr]
&&{\begin{smallmatrix}
&2&\ \\
3&&4
\end{smallmatrix}}\ar[ur] \ar@{.}[rr]
&&{\begin{smallmatrix}
\ &2&\
\end{smallmatrix}}\\
&&&{\begin{smallmatrix}
\ &3&\
\end{smallmatrix}} \ar[ur]}$$
Since $\U_1\subseteq \M\subseteq \V_1$, we have $\W_1\subseteq \K_2\subseteq\K_1$.
Since We get $H_1({\begin{smallmatrix}
4&&3\\
&5&
\end{smallmatrix}})={\begin{smallmatrix}
3&&\\
&5&
\end{smallmatrix}}$, we get that $\beta_{21}$ is exact. But $\beta_{12}$ is not exact, since $\xymatrix@C=0.4cm@R0.4cm{
{\begin{smallmatrix}
&3&\ \\
&&5
\end{smallmatrix}} \ar[r] &{\begin{smallmatrix}
\ &3&\
\end{smallmatrix}} \ar[r] &{\begin{smallmatrix}
&2&\ \\
3&&4
\end{smallmatrix}}}$ is a short exact sequence in $\h_1/\W_1$ but not a short exact sequence in $\h_2/\W_2$. In this case, $(\h_2\cap\K_1/\W_2)$ is $\add({\begin{smallmatrix}
\ &5&\
\end{smallmatrix}})$, we can see that $\overline \h_2\simeq \h_1/\W_1$.\\
Let $\U_3$ and $\V_3$ be the full subcategories of $\mod \Lambda$ given by the following diagram.
$$
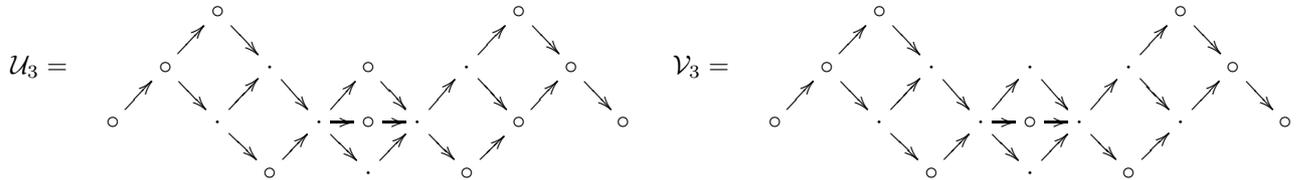
\xymatrix@C=0.3cm@R0.3cm{
 &&&\circ \ar[dr] &&&&&&\circ \ar[dr]\\
{\U_3=} &&\circ \ar[ur]  \ar[dr] &&\cdot  \ar[dr] &&\circ  \ar[dr] &&\cdot  \ar[ur]  \ar[dr] &&\circ \ar[dr]\\
&\circ \ar[ur]  &&\cdot\ar[ur]  \ar[dr] &&\cdot \ar[ur] \ar[r] \ar[dr] &\circ \ar[r] &\cdot \ar[ur] \ar[dr] &&\circ \ar[ur] &&\circ\\
&&&&\circ \ar[ur] &&\cdot \ar[ur] &&\circ \ar[ur]\\}\quad
\xymatrix@C=0.3cm@R0.3cm{
&&&\circ \ar[dr] &&&&&&\circ \ar[dr]\\
{\V_3=} &&\circ \ar[ur]  \ar[dr] &&\cdot  \ar[dr] &&\cdot  \ar[dr] &&\cdot  \ar[ur]  \ar[dr] &&\circ \ar[dr]\\
&\circ \ar[ur]  &&\cdot \ar[ur]  \ar[dr] &&\cdot \ar[ur] \ar[r] \ar[dr] &\circ \ar[r] &\cdot \ar[ur] \ar[dr] &&\cdot \ar[ur] &&\circ\\
&&&&\circ \ar[ur] &&\cdot \ar[ur] &&\circ \ar[ur]}$$
and the heart $\h_3/\W_3$ is the following.
$$
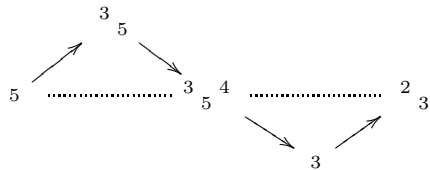
\xymatrix@C=0.4cm@R0.4cm{
&{\begin{smallmatrix}
&3&\ \\
&&5
\end{smallmatrix}}\ar[dr]\\
{\begin{smallmatrix}
\ &5&\
\end{smallmatrix}} \ar[ur] \ar@{.}[rr]
&&{\begin{smallmatrix}
3&&4\ \\
&5&
\end{smallmatrix}}\ar[dr] \ar@{.}[rr]
&&{\begin{smallmatrix}
&2&\ \\
&&3
\end{smallmatrix}}\\
&&&{\begin{smallmatrix}
\ &3&\
\end{smallmatrix}} \ar[ur]}$$
Hence we get $\h_1/\W_1\simeq \h_3/\W_3$. But we find that $\U_3\nsubseteqq\K_1$ and $\V_1\nsubseteqq\add(\U_3*\V_3)$, which implies that the condition Corollary \ref{suf} is not necessary for the equivalence of two hearts.
\end{exm}

By Theorem \ref{8.11} and Proposition \ref{8.15}, we get:

\begin{prop}\label{5.00}
Let $\M$ be a cluster tilting subcategory of $\B$. Then the canonical functor
$$\pi:\B\rightarrow \B/\M$$
is half exact. Moreover, every short exact sequence
$$\xymatrix{A \ar@{ >->}[r]^{f} &B \ar@{->>}[r]^{g} &C}$$
in $B$ induces a long exact sequence
%\begin{align*}
$$\cdots \xrightarrow{\underline {\Omega h'}} \Omega A \xrightarrow{\underline {\Omega f}} \Omega B \xrightarrow{\underline {\Omega g}} \Omega C \xrightarrow{\underline h'} A \xrightarrow{\underline f} B \xrightarrow{\underline g} C \xrightarrow{\underline h} \Omega^- A \xrightarrow{\underline {\Omega^- f}} \Omega^- B \xrightarrow{\underline {\Omega^- g}} \Omega^- C \xrightarrow{\underline {\Omega^- h}} \cdots$$
%\end{align*}
in the abelian category $\B/\M$.
\end{prop}

%The following is a corollary of Proposition \ref{5.00}.

\begin{exm}
Let $\M$ be a cluster tilting subcategory of $\B$ (for instance, see \cite[{Example 4.2}]{DL}). Then we have a half exact functor
\begin{align*}
G: \text{ }&\B \rightarrow \mod(\M/\mathcal P)\\
   &X\mapsto \Ext^1_{\B}(-,X)|_{\M}.
\end{align*}
This is a composition of the half exact functor $\pi:\B\rightarrow \B/\M$ given by Proposition \ref{5.00} and an equivalence
\begin{align*}
 \text{ }&\B/\M \xrightarrow{\simeq} \mod(\M/\mathcal P)\\
   &X\mapsto \Ext^1_{\B}(-,X)|_{\M}.
\end{align*}
given by \cite[{Theorem 3.2}]{DL}. By Proposition \ref{8.14}, $G(X)=0$ if and only if $X\in \M$.
\end{exm}

A more general case is given as follows. If $\M$ is a rigid subcategory of $\B$ which is contravariantly finite and contains $\mathcal P$, then by \cite[{Proposition 2.12}]{L}, $(\M,{\M}^{\bot_1})$ is a cotorsion pair where ${\M}^{\bot_1}=\{X\in \B \text{ }| \text{ }\Ext^1_\B(\M,X)=0 \}$. Since $\M$ is rigid, we have $\M \subseteq {\M}^{\bot_1}$. In this case we have %$\mathcal P,\mathcal I \subseteq \M$
$\B^+=\B$, $\B^-=\h$, $\sigma^+=\text{id}$ and $H=\sigma^-\circ \pi$. By \cite[{Theorem 3.2}]{DL}, there exists an equivalence between $\underline \h$ and $\mod(\M/\mathcal P)$. Hence by Theorem \ref{8.11}, we get the following example:

\begin{exm}
Let $\M$ be a rigid subcategory of $\B$ which is contravariantly finite and contains $\mathcal P$ (for instance, see \cite[{Example 4.3}]{DL}). Then there exists a half exact functor
\begin{align*}
G: \text{ }&\B \rightarrow \mod(\M/\mathcal P)\\
   &X\mapsto \Ext^1_{\B}(-,\sigma^-(X))|_{\M}
\end{align*}
which is a composition of $H$ and the equivalence
\begin{align*}
\text{ }&\underline \h \xrightarrow{\simeq} \mod(\M/\mathcal P)\\
   &Y\mapsto \Ext^1_{\B}(-,Y)|_{\M}
\end{align*}
given by \cite[{Theorem 3.2}]{DL}. By Proposition \ref{8.14}, $G(X)=0$ if and only if $X\in {\M}^{\bot_1}$.
\end{exm}

%Let $\B$ be a Frobenius category with a cotorsion pair $(\U,\V)$ on it. %we have $\mathcal I=\mathcal P\subseteq \W$.

%\begin{cor}
%For any cotorsion pair $(\U,\V)$ on a Frobenius category $\B$, there exists a cohomological functor $H'$ from the stable category of $\B$ to $\underline \h$ which is induced by $H$.
%\end{cor}

%\begin{proof}
%Let $\pi':\B\rightarrow \B/\mathcal P$ be the quotient functor, since $H(\mathcal P)=0$ by Proposition \ref{8.7}, there exists a functor $H':\B/\mathcal P \rightarrow \underline \h$ such that $H=H'\pi'$. We shall prove that $H'$ is cohomological. Since every triangle in $\B/\mathcal P$ is isomorphic to some distinguished triangle
%$$X\xrightarrow{\pi'(f)} Y\xrightarrow{\pi'(g)} Z\rightarrow \Omega^-X$$
%which is induced by the following commutative diagram
%$$\xymatrix{
%X \ar@{ >->}[r]^p \ar[d]_f \ar@{}[dr]|{PO} &I^X \ar@{->>}[r] \ar[d]^q &\Omega^-X \ar@{=}[d]\\
%Y \ar@{ >->}[r]_g &Z \ar@{->>}[r] &\Omega^-X
%}
%$$
%it suffices to consider such distinguished triangle. But such commutative diagram induce a short exact sequence
%$$\xymatrix{X \ar@{ >->}[r]^-{\svecv{-p}{f}} &I^X\oplus Y \ar@{->>}[r]^-{\svech{q}{g}} &Z.}$$
%By Theorem \ref{8.11}, we get a exact sequence
%$$H(X) \xrightarrow{H(f)} H(Y) \xrightarrow{H(g)} H(Z)$$
%in $\underline \h$, which also has the form
%$$H'(X) \xrightarrow{H'\pi'(f)} H'(Y) \xrightarrow{H'\pi'(g)} H'(Z).$$
%\end{proof}

\section*{Acknowledgments}
The author would like to thank Osamu Iyama for his helpful advices and corrections.

%This $\beta_k$ is not always full or dense, since there are examples in which the heart of twin cotorsion pair is zero but the hearts of its two components are not (see \cite[Example 10.10]{L} for details).

\end{document}